\documentclass{amsart}

\title{Modal model theory}

\author{Joel David Hamkins}
\address[Joel David Hamkins]
{Professor of Logic, University of Oxford \&\ Sir Peter Strawson Fellow, University College, High Street, Oxford OX1 4BH, United Kingdom}
         \email{joeldavid.hamkins@philosophy.ox.ac.uk}
         \urladdr{http://jdh.hamkins.org}

\author{Wojciech Aleksander Wo\l oszyn}
\address[Wojciech Aleksander Wo\l oszyn]
{Mathematical Institute, University of Oxford, Andrew Wiles Building, Radcliffe Observatory Quarter, Woodstock Road, Oxford, OX2 6GG, United Kingdom \&\ St Hilda's College, Cowley Place, Oxford, OX4 1DY, United Kingdom}
\email{wojciech.woloszyn@st-hildas.ox.ac.uk}
\urladdr{https://woloszyn.org}

\thanks{Commentary can be made about this article on the first author's blog at \href{http://jdh.hamkins.org/modal-model-theory}{http://jdh.hamkins.org/modal-model-theory}.}

\usepackage{latexsym,amsfonts,amsmath,amssymb,mathrsfs}
\usepackage{bbm}
\usepackage[hidelinks]{hyperref}
\usepackage{relsize}
\usepackage[rgb,dvipsnames]{xcolor}
\usepackage{tikz}
\usepackage{pgflibrarysnakes}
\usetikzlibrary{arrows,arrows.meta,petri,topaths,positioning,shapes,shapes.misc,patterns,calc,decorations.pathreplacing,hobby,snakes}
\usepackage{wrapfig} 
\usepackage{float}
\usepackage[utf8]{inputenc}
\RequirePackage{doi}

\usepackage{enumitem}

\newtheorem{theorem}{Theorem}
\newtheorem*{theorem*}{Theorem}

\newtheorem*{maintheorem*}{Main Theorem}
\newtheorem*{maintheorems*}{Main Theorems}
\newtheorem{corollary}[theorem]{Corollary}
\newtheorem*{corollary*}{Corollary}
\newtheorem*{corollaries*}{Corollaries}

\newtheorem{keylemma}[theorem]{Key Lemma}

\newtheorem{question}[theorem]{Question}
\newtheorem*{question*}{Question}

\newtheorem*{questions*}{Questions}
\newtheorem*{mainquestion*}{Main Question} 
\newtheorem*{openquestion*}{Open Question} 
\newtheorem{observation}[theorem]{Observation}

\newtheorem{definition}[theorem]{Definition}

\newcommand{\QED}{\end{proof}}

\def\proclaim[#1]{{\bf #1}}
\def\BF#1.{{\bf #1.}}

\def\says#1:#2\par{\item[#1] #2\par}

\newcommand{\Los}{\L o\'s}

\newcommand{\Godel}{G\"odel}

\newcommand{\Lowenheim}{L\"owenheim}

\newcommand{\N}{{\mathbb N}}

\newcommand{\Z}{{\mathbb Z}}

\newcommand{\continuum}{\mathfrak{c}}

\newcommand{\overbar}[1]{\mkern 3.5mu\overline{\mkern-3.5mu#1\mkern-.5mu}\mkern.5mu}

\newcommand\p{\frak{p}}

\makeatletter
\newcommand{\dotminus}{\mathbin{\text{\@dotminus}}}
\newcommand{\@dotminus}{%
  \ooalign{\hidewidth\raise1ex\hbox{.}\hidewidth\cr$\m@th-$\cr}%
}
\makeatother

\newcommand{\of}{\subseteq}

\newcommand{\sqof}{\sqsubseteq}

\newcommand{\singleton}[1]{\left\{{#1}\right\}}

\newcommand{\elesub}{\prec}
\newcommand{\eleequiv}{\equiv}

\newcommand{\Th}{\mathop{\rm Th}}

\newcommand{\Mod}{\mathop{{\rm Mod}}}
\newcommand{\image}{\mathbin{\hbox{\tt\char'42}}}

\newcommand{\satisfies}{\models}

\DeclareMathOperator{\possible}{\text{\tikz[scale=.6ex/1cm,baseline=-.6ex,rotate=45,line width=.1ex]{\draw (-1,-1) rectangle (1,1);}}}
\DeclareMathOperator{\necessary}{\text{\tikz[scale=.6ex/1cm,baseline=-.6ex,line width=.1ex]{\draw (-1,-1) rectangle (1,1);}}}

\newcommand{\axiomf}[1]{{\rm #1}}
\newcommand{\theoryf}[1]{{\rm #1}}

\newcommand{\union}{\cup}

\newcommand{\Union}{\bigcup}

\newcommand{\smalllt}{\mathrel{\mathchoice{\raise2pt\hbox{$\scriptstyle<$}}{\raise1pt\hbox{$\scriptstyle<$}}{\raise0pt\hbox{$\scriptscriptstyle<$}}{\scriptscriptstyle<}}}
\newcommand{\smallleq}{\mathrel{\mathchoice{\raise2pt\hbox{$\scriptstyle\leq$}}{\raise1pt\hbox{$\scriptstyle\leq$}}{\raise1pt\hbox{$\scriptscriptstyle\leq$}}{\scriptscriptstyle\leq}}}

   \def\DHLhksqrt#1#2{%
   \setbox0=\hbox{$#1\sqrt{#2\,}$}\dimen0=\ht0
   \advance\dimen0-0.2\ht0
   \setbox2=\hbox{\vrule height\ht0 depth -\dimen0}%
   {\box0\lower0.4pt\box2}}
\newcommand{\boolval}[1]{\mathopen{\lbrack\!\lbrack}\,#1\,\mathclose{\rbrack\!\rbrack}}
\def\[#1]{\boolval{#1}}
\newbox\gnBoxA
\newbox\gnBoxB
\newdimen\gnCornerHgt
\setbox\gnBoxA=\hbox{\tiny$\ulcorner$}
\global\gnCornerHgt=\ht\gnBoxA
\newdimen\gnArgHgt
\def\gcode #1{%
\setbox\gnBoxA=\hbox{$#1$}%
\setbox\gnBoxB=\hbox{$\bar #1$}%
\gnArgHgt=\ht\gnBoxB%
\ifnum     \gnArgHgt<\gnCornerHgt \gnArgHgt=0pt%
\else \advance \gnArgHgt by -\gnCornerHgt%
\fi \raise\gnArgHgt\hbox{\tiny$\ulcorner$} \box\gnBoxA %
\raise\gnArgHgt\hbox{\tiny$\urcorner$}}
\newcommand{\UnderTilde}[1]{{\setbox1=\hbox{$#1$}\baselineskip=0pt\vtop{\hbox{$#1$}\hbox to\wd1{\hfil$\sim$\hfil}}}{}}
\newcommand{\Undertilde}[1]{{\setbox1=\hbox{$#1$}\baselineskip=0pt\vtop{\hbox{$#1$}\hbox to\wd1{\hfil$\scriptstyle\sim$\hfil}}}{}}
\newcommand{\undertilde}[1]{{\setbox1=\hbox{$#1$}\baselineskip=0pt\vtop{\hbox{$#1$}\hbox to\wd1{\hfil$\scriptscriptstyle\sim$\hfil}}}{}}
\newcommand{\UnderdTilde}[1]{{\setbox1=\hbox{$#1$}\baselineskip=0pt\vtop{\hbox{$#1$}\hbox to\wd1{\hfil$\approx$\hfil}}}{}}
\newcommand{\Underdtilde}[1]{{\setbox1=\hbox{$#1$}\baselineskip=0pt\vtop{\hbox{$#1$}\hbox to\wd1{\hfil\scriptsize$\approx$\hfil}}}{}}

\renewcommand{\implies}{\mathrel{\rightarrow}}

\renewcommand{\iff}{\mathrel{\leftrightarrow}}
\newcommand{\Iff}{\mathrel{\Longleftrightarrow}}

\def\<#1>{\left\langle#1\right\rangle}

\newcommand{\TC}{\mathop{{\rm TC}}}

\newcommand{\ETR}{{\rm ETR}}

\newcommand{\ZFC}{{\rm ZFC}}

\newcommand{\ZFCm}{\ZFC^-}

\newcommand{\KM}{{\rm KM}}

\newcommand{\GBC}{{\rm GBC}}

\newcommand{\cell}[1]{\boxit{\hbox to 17pt{\strut\hfil$#1$\hfil}}}
\newcommand{\head}[2]{\lower2pt\vbox{\hbox{\strut\footnotesize\it\hskip3pt#2}\boxit{\cell#1}}}
\newcommand{\boxit}[1]{\setbox4=\hbox{\kern2pt#1\kern2pt}\hbox{\vrule\vbox{\hrule\kern2pt\box4\kern2pt\hrule}\vrule}}
\newcommand{\Col}[3]{\hbox{\vbox{\baselineskip=0pt\parskip=0pt\cell#1\cell#2\cell#3}}}
\newcommand{\tapenames}{\raise 5pt\vbox to .7in{\hbox to .8in{\it\hfill input: \strut}\vfill\hbox to
.8in{\it\hfill scratch: \strut}\vfill\hbox to .8in{\it\hfill output: \strut}}}
\newcommand{\Head}[4]{\lower2pt\vbox{\hbox to25pt{\strut\footnotesize\it\hfill#4\hfill}\boxit{\Col#1#2#3}}}
\newcommand{\Dots}{\raise 5pt\vbox to .7in{\hbox{\ $\cdots$\strut}\vfill\hbox{\ $\cdots$\strut}\vfill\hbox{\
$\cdots$\strut}}}
\tikzset{>=Stealth,
  dot/.style={circle,draw,fill,inner sep=#1},
    dot/.default=.7pt
  }
\pgfdeclarelayer{background}
\pgfdeclarelayer{background0}
\pgfdeclarelayer{background1}
\pgfdeclarelayer{background2}
\pgfsetlayers{background2,background1,background0,background,main}

\DeclareMathOperator{\embedpossible}{\text{\tikz[scale=.6ex/1cm,baseline=-.6ex,rotate=45,line width=.1ex]{\draw (-1,-1) rectangle (1,1);
    \draw (-1,.67) to[looseness=2.2,out=0,in=180] (1,-.67);}}}
\DeclareMathOperator{\embednecessary}{\text{\tikz[scale=.6ex/1cm,baseline=-.6ex,line width=.1ex]{\draw (-1,-1) rectangle (1,1);
    \draw (-1,0) to[looseness=2,out=40,in=220] (1,0);}}}
\newcommand{\lperpdown}{\mathrel{\text{\tikz[rotate=20,scale=2.5ex/1cm,baseline=-.1ex,line width=.125ex]{\draw (0,0) -- (1,0) (.67,0) -- (.67,-.33);}}}}
\newcommand{\ofsim}{\mathrel{\raisebox{-3pt}{\vbox{\baselineskip=3pt\hbox{$\subset$}\hbox{\tiny$\,\sim$}}}}}

\begin{document}

\begin{abstract}
We introduce the subject of modal model theory, where one studies a mathematical structure within a class of similar structures under an extension concept that gives rise to mathematically natural notions of possibility and necessity. A statement $\varphi$ is \emph{possible} in a structure (written $\possible\varphi$) if $\varphi$ is true in some extension of that structure, and $\varphi$ is \emph{necessary} (written $\necessary\varphi$) if it is true in all extensions of the structure. A principal case for us will be the class $\Mod(T)$ of all models of a given theory $T$---all graphs, all groups, all fields, or what have you---considered under the substructure relation. In this article, we aim to develop the resulting modal model theory. The class of all graphs is a particularly insightful case illustrating the remarkable power of the modal vocabulary, for the modal language of graph theory can express connectedness, $k$-colorability, finiteness, countability, size continuum, size $\aleph_1$, $\aleph_2$, $\aleph_\omega$, $\beth_\omega$, first $\beth$-fixed point, first $\beth$-hyper-fixed-point and much more. A graph obeys the maximality principle $\possible\necessary\varphi(a)\to\varphi(a)$ with parameters if and only if it satisfies the theory of the countable random graph, and it satisfies the maximality principle for sentences if and only if it is universal for finite graphs.
\end{abstract}

\maketitle

\tableofcontents

\section{Introduction}

In modal model theory, we consider a mathematical structure within the context of a class of similar structures, investigating the nature of possibility and necessity as one extends to larger structures. In the general case, we have a \emph{potentialist system}, which is a class $\mathcal W$ of models in a common language together with an extension relation $M\sqof N$ refining the substructure relation, and we define the natural modal operators:
\begin{enumerate}
  \item A model $M$ in $\mathcal W$ thinks $\varphi$ is \emph{possible}, written $M\satisfies\possible\varphi$, if there is an extension $M\sqof N$ with $N\satisfies \varphi$.
  \item A model $M$ in $\mathcal W$ thinks $\varphi$ is \emph{necessary}, written $M\satisfies\necessary\varphi$, if every extension $M\sqof N$ has $N\satisfies \varphi$.
\end{enumerate}
In modal model theory, we focus particularly on the case of $\mathcal{W}=\Mod(T)$, the potentialist system consisting of all the models of a given first-order theory $T$, considered under the submodel relation. We shall consider all graphs, all groups, all fields, or what have you, considering each model in the context of all models of the corresponding theory. By augmenting the languages with these modal operators, every first-order language and theory thus extends canonically to a modal language and theory, whose basic model theory and expressive power we aim to investigate.

To illustrate the modal vocabulary, observe that in the class of all graphs, every graph thinks ``possibly the diameter is $2$,'' since any graph can be extended to include a new vertex having suitable edges with the previous nodes so as to make the diameter of the larger graph exactly $2$; in the class of all groups, every group is possibly necessarily non-abelian, since every group is a subgroup of a nonabelian group, all of whose further extensions will be non-abelian; and in fields, possibly every element is a square, but this statement is necessarily not necessary.


For clarity let us be precise about the various distinct but closely related languages we shall treat in this article.
\begin{enumerate}
  \item We denote by $\mathcal{L}$ the language of the structures in the potentialist system $\mathcal{W}$. In the case of $\Mod(T)$, this is the language of the theory $T$.
  \item $\possible\mathcal{L}$ is the closure of $\mathcal{L}$ under the modal operators $\possible,\necessary$ and Boolean connectives (but not quantifiers).
  \item $\mathcal{L}^{\possible}$ is the full first-order modal language, closing $\mathcal{L}$ under modal operators, Boolean connectives and quantifiers.
  \item $\mathcal{L}^{\possible,@}$ extends the full modal language with the actuality operator~@, explained in section \ref{Section.Actuality}.
  \item $\mathcal{P}$ is the language of propositional modal logic, with propositional variables $p$, $q$, $r$, and so on, closed under Boolean connectives and modal operators.
\end{enumerate}
The assertions of $\possible\mathcal{L}$ are exactly those assertions of $\mathcal{L}^{\possible}$ in which no modal operator falls under the scope of a quantifier; these are the same as the substitution instances $\varphi(\psi_0,\ldots,\psi_n)$, where $\varphi(p_0,\ldots,p_n)$ is a propositional modal assertion in $\mathcal{P}$ and each $\psi_i$ is an assertion of the first-order language $\mathcal{L}$. Our results in section \ref{Section.Elementary-modal-model-theory} will show that much of the classical model theory extends from the base language $\mathcal{L}$ to $\possible\mathcal{L}$, but not to $\mathcal{L}^{\possible}$, which is why we separate out this language fragment. For example, theorem \ref{Theorem.Lowenheim-Skolem} shows that $\possible\mathcal{L}$ obeys the \Lowenheim-Skolem theorem in $\Mod(T)$, but $\mathcal{L}^{\possible}$ does not in general.

The \emph{potentialist system} terminology was introduced in \cite{HamkinsLinnebo:Modal-logic-of-set-theoretic-potentialism}, with additional related work in \cite{HamkinsWoodin:The-universal-finite-set, HamkinsWilliams:The-universal-finite-sequence, Hamkins:The-modal-logic-of-arithmetic-potentialism} analyzing the modal logic of models of arithmetic and set theory. The modal analysis of forcing extensions in set theory had begun earlier with \cite{Hamkins2003:MaximalityPrinciple} and continued with \cite{HamkinsLoewe2008:TheModalLogicOfForcing, HamkinsLoewe2013:MovingUpAndDownInTheGenericMultiverse, HamkinsLeibmanLoewe2015:StructuralConnectionsForcingClassAndItsModalLogic}.

\section{The remarkable expressive power of modal graph theory}\label{Section.Expressive-power-of-modal-graph-theory}

Let us begin our project by illustrating the remarkable expressive power of the modal language of graph theory. We work in the class of all graphs, where a graph is a set of vertices and an irreflexive, symmetric binary edge relation $\sim$. (So, there are no self-edges and no parallel edges.) Let $\mathcal{L}_\sim$ be the first-order language of graph theory, which has only the binary edge relation $\sim$.

\begin{theorem}\label{Theorem.2-colorability-is-expressible}
In the class of graphs, $2$-colorability is expressible in the modal language of graph theory. There is a sentence $\rho_2\in \possible\mathcal{L}_\sim$ such that for any graph $G$,
 $$G\satisfies\rho_2\qquad\text{ if and only if }\qquad G\text{ is $2$-colorable.}$$
Similarly, $k$-colorability is expressible by a sentence $\rho_k$ for any finite $k$.
\end{theorem}

\begin{proof}
We claim that a graph $G$ is $2$-colorable if and only if possibly, there are adjacent nodes $r$ and $b$, such that every node is adjacent to exactly one of them and adjacent nodes are connected to them oppositely.
$$
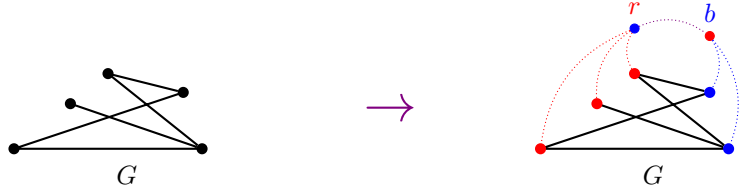
\begin{tikzpicture}[line join=bevel,dot/.default=1.2pt]
  \draw[thick]  (2.5,0) node[dot] (b1) {} --
         (0,0) node[dot] (r1) {} --
         (2.25,.75) node[dot] (b2) {}
         (.75,.6) node[dot] (r2) {} -- (b1) --
         (1.25,1) node[dot] (r3) {} -- (b2)
         (1.5,-.35) node {$G$};
  \draw (5,.5) node[scale=2,violet] {$\to$};
  \begin{scope}[xshift=7cm]
  \draw[thick]  (2.5,0) node[dot,blue] (b1) {} --
         (0,0) node[dot,red] (r1) {} --
         (2.25,.75) node[dot,blue] (b2) {}
         (.75,.6) node[dot,red] (r2) {} -- (b1) --
         (1.25,1) node[dot,red] (r3) {} -- (b2)
         (1.5,-.35) node {$G$};
    \draw (1.25,1.6) node[blue,dot,label={[red]above:$r$}] (r) {}
          (2.25,1.5) node[red,dot,label={[blue]above:$b$}] (b) {};
    \draw[densely dotted,red!50!blue,bend left] (r) to (b);
    \draw[densely dotted,red,bend right]
        (r) to (r1)
        (r) to (r2)
        (r) to (r3);
    \draw[densely dotted,blue,bend left]
        (b) to (b1)
        (b) to (b2);
  \end{scope}
\end{tikzpicture}$$
If we think of the neighbors of $r$ as red and the neighbors of $b$ as blue, any such graph extension will be $2$-colorable. (Note that this coloring assignment will actually color $r$ blue and $b$ red, since they are neighbors of each other and not of themselves.) Conversely, if the graph is $2$-colorable, then we can add new vertices $r$ and $b$ and join with edges according to the coloring. A similar idea works for $k$-colorings with any finite number $k$.
\end{proof}

One might notice that in the particular figure of the proof, we needn't actually have added any nodes at all, since we could have used the bottom two nodes as $b$ and $r$ to realize the property already in the original graph.

\begin{theorem}
In the class of graphs, connectivity of nodes is expressible in the modal language of graph theory. There is a formula $\chi(x,y)$ in $\possible\mathcal{L}_\sim$ expressing that vertex $x$ is connected to vertex $y$.
 $$G\satisfies\chi(x,y)\qquad\text{ if and only if }\qquad x\text{ is connected to }y\text{ in }G.$$
Similarly, there is a sentence in $\mathcal{L}^{\possible}_\sim$ expressing that the graph as a whole is connected.
\end{theorem}

\begin{proof}
Vertex $x$ is connected with $y$ if and only if there is a finite path in the graph from $x$ to $y$. This is equivalent to asserting in the modal language that necessarily, any vertex $c$ that is adjacent to $x$ and whose neighbors are closed under adjacency is also adjacent to $y$. This can be expressed in the modal language of graph theory by the formula $\chi(x,y)$ as follows:
$$\necessary\forall c[(c\sim x\wedge\forall u,v (c\sim u\wedge u\sim v\wedge v\neq c\to c\sim v))\to c\sim y].$$
To illustrate, consider the graph pictured here at left in black, where vertex $x$ is in fact connected to vertex $y$.
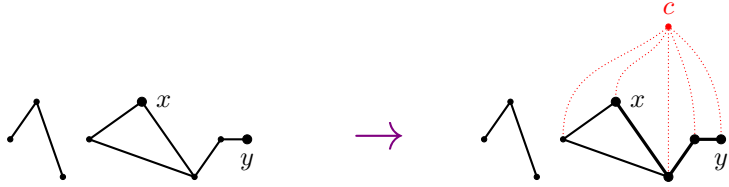
$$\begin{tikzpicture}[xscale=.7]
\begin{scope}
  \node[dot] (a) at (-1.5,.5) {};
  \node[dot] (b) at (-1,1) {};
  \node[dot] (c) at (-.5,0) {};
  \node[dot] (d) at (0,.5) {};
  \node[dot=1.2pt,label=right:$x$] (e) at (1,1) {};
  \node[dot] (f) at (2,0) {};
  \node[dot] (g) at (2.5,.5) {};
  \node[dot=1.2pt,label=below:$y$] (h) at (3,.5) {};
  \draw[thick] (a.center) -- (b.center) -- (c.center)
        (f.center) -- (d.center) -- (e.center) -- (f.center) -- (g.center) -- (h.center);
  \draw (5.5,.5) node[scale=2,violet] {$\to$};
\end{scope}
\begin{scope}[xshift=9cm]
  \node[dot] (a) at (-1.5,.5) {};
  \node[dot] (b) at (-1,1) {};
  \node[dot] (c) at (-.5,0) {};
  \node[dot] (d) at (0,.5) {};
  \node[dot=1.2pt,label=right:$x$] (e) at (1,1) {};
  \node[dot=1.2pt] (f) at (2,0) {};
  \node[dot=1.2pt] (g) at (2.5,.5) {};
  \node[dot=1.2pt,label=below:$y$] (h) at (3,.5) {};
  \draw[thick] (a.center) -- (b.center) -- (c.center)
        (f.center) -- (d.center) -- (e.center);
  \draw[very thick] (e.center) -- (f.center) -- (g.center) -- (h.center);
  \node[dot,red,label={[red]above:$c$}] (c) at (2,2) {};
  \draw[red,densely dotted] \foreach \p [count=\i] in {d,e,f,g,h} { (c) to[in=90,out=180+30*\i] (\p) };
\end{scope}
\end{tikzpicture}$$
Consider an extension as at right with a vertex $c$ that is adjacent to $x$, whose neighbors are closed under adjacency. Inductively $c$ will be adjacent to every vertex on the path from $x$ to $y$, and consequently it will be adjacent to $y$. Conversely, in any graph where $x$ is not connected to $y$, we may extend it to a graph with a new node $c$ adjacent to every vertex in the connected component of $x$ and to no others. This node will be adjacent to $x$ and its neighbors will be closed under adjacency, but it will not be adjacent to $y$. So the formula will hold in a graph precisely when $x$ is connected to $y$. For a nonempty graph as a whole to be connected, it should satisfy $\forall x\forall y\ \chi(x,y)$.
\end{proof}

Notice that by relativizing the expression, we can express the connectivity of any particular definable subgraph of a graph, such as the set of neighbors of a given node.

It is a standard model-theoretic exercise to show that neither $2$-colorability nor connectivity for graphs is expressible in the ordinary language of graph theory, and so these observations show that the expressive power of modal graph theory $\possible\mathcal{L}_\sim$ strictly exceeds first-order graph theory. Let us continue---far more is expressible with the modal vocabulary.

\begin{theorem}\label{Theorem.Finiteness-is-expressible}
In the class of graphs, finiteness is expressible in the modal language of graph theory. There is a sentence $\phi\in\mathcal{L}_\sim^{\possible}$ such that for any graph $G$,
 $$G\satisfies\phi\qquad\text{ if and only if }\qquad G\text{ is finite.}$$
\end{theorem}

\begin{proof}
A graph $G$ is finite if and only if possibly, there is a point $n$, whose neighbor graph is connected, with all vertices of degree $2$ within that neighbor set, except exactly two vertices of degree $1$ in that neighbor set---a starting node and an ending node---and all other nodes of the graph are adjacent to exactly one neighbor of $n$ in a bijective correspondence. The neighbors of $n$ will form a finite chain from the starting vertex to the ending vertex, and the bijection will show that the graph is finite.
$$
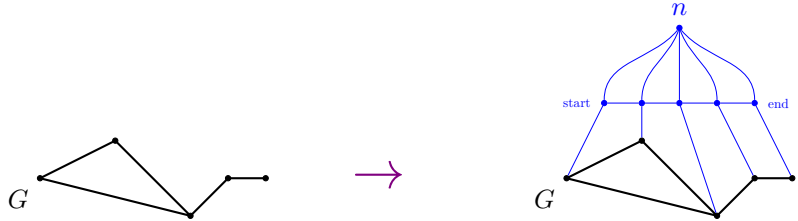
\begin{tikzpicture}
  \node[dot,label=below left:$G$] (d) at (0,.5) {};
  \node[dot] (e) at (1,1) {};
  \node[dot] (f) at (2,0) {};
  \node[dot] (g) at (2.5,.5) {};
  \node[dot] (h) at (3,.5) {};
  \draw[thick] (f.center) -- (d.center) -- (e.center) -- (f.center) -- (g.center) -- (h.center);
  \draw (4.5,.5) node[scale=2,violet] {$\to$};
\begin{scope}[xshift=7cm]
  \node[dot,label=below left:$G$] (d) at (0,.5) {};
  \node[dot] (e) at (1,1) {};
  \node[dot] (f) at (2,0) {};
  \node[dot] (g) at (2.5,.5) {};
  \node[dot] (h) at (3,.5) {};
  \draw[thick] (f.center) -- (d.center) -- (e.center) -- (f.center) -- (g.center) -- (h.center);
  \node[dot,blue,label={[blue]above:$n$}] (n) at (1.5,2.5) {};
  \foreach \k in {0,...,4} {\node[dot,blue] (\k) at (.5+.5*\k,1.5) {}; \draw[blue,very thin] (n) to[in=90,out=-130+20*\k] (\k);}
  \draw[blue,very thin,every label/.style={scale=.5}] (0) node[label=left:start] {} -- (1) -- (2) -- (3) -- (4) node[label=right:end] {};
  \draw[blue,very thin] (d) -- (0)
                             (e) -- (1)
                             (f) -- (2)
                             (g) -- (3)
                             (h) -- (4);
\end{scope}
\end{tikzpicture}$$
Since the neighbor graph of vertex $n$ is connected, it cannot have an infinite number of neighbors, for when proceeding successively from the \emph{start} node through the intermediate nodes of degree $2$, one must come to the \emph{end} node in finitely many steps, for otherwise the \emph{start} and \emph{end} nodes will occupy different connected components, contrary to assumption. So $n$ really does have only finitely many neighbors under those assumptions, and so the graph will be finite.

Note that when employing the ``possibly'' operator, the original graph also may get larger, and so the bijection is not necessarily talking only about the original graph in the extension with $n$, but possibly an extension of the original graph.
$$
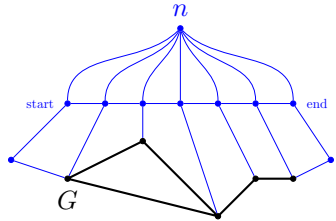
\begin{tikzpicture}
\begin{scope}[xshift=5cm]
  \node[dot,label=below:$G$] (d) at (0,.5) {};
  \node[dot] (e) at (1,1) {};
  \node[dot] (f) at (2,0) {};
  \node[dot] (g) at (2.5,.5) {};
  \node[dot] (h) at (3,.5) {};
  \draw[thick] (f.center) -- (d.center) -- (e.center) -- (f.center) -- (g.center) -- (h.center);
  \node[dot,blue] (c) at (-.75,.75) {};
  \node[dot,blue] (i) at (3.5,.75) {};
  \node[dot,blue,label={[blue]above:$n$}] (n) at (1.5,2.5) {};
  \foreach \k in {-1,...,5} {\node[dot,blue] (\k) at (.5+.5*\k,1.5) {}; \draw[blue,very thin] (n) to[in=90,out=-120+15*\k] (\k);}
  \draw[blue,very thin,every label/.style={scale=.5}] (-1) node[label=left:start] {} -- (0) -- (1) -- (2) -- (3) -- (4) -- (5) node[label=right:end] {};
  \draw[blue,very thin] (c) -- (-1)
                             (d) -- (0)
                             (e) -- (1)
                             (f) -- (2)
                             (g) -- (3)
                             (h) -- (4)
                             (i) -- (5)
                             (c) -- (d)
                             (h) -- (i);
\end{scope}
\end{tikzpicture}$$
But since a set is finite if and only if it has a finite extension, this is not a problem for the assertion.
\end{proof}

\begin{corollary}
In the class of graphs, for a vertex to have finite degree is a property expressible in the modal language of graph theory. There is a formula $\eta(x)$ in $\mathcal{L}_\sim^{\possible}$ that holds of a node $x$ in a graph $G$ if and only if $x$ has finite degree in $G$.
\end{corollary}

\begin{proof}
This is a consequence of the previous theorem, by relativizing to the subgraph consisting of the neighbors of $x$. That is, $\eta(x)$ asserts that possibly, there is a node $n$ whose neighbors are connected and all of degree $2$ except a starting node and an ending node, such that every neighbor of $x$ is adjacent to a distinct neighbor of $n$.
\end{proof}

\begin{theorem}\label{Theorem.Countability-is-expressible}
 In the class of graphs, countability is expressible in the modal language of graph theory. There is a sentence $\sigma\in \mathcal{L}_\sim^{\possible}$ such that for any graph $G$,
 $$G\satisfies\sigma\qquad\text{ if and only if }\qquad G\text{ is countable.}$$
\end{theorem}

\begin{proof}
A graph $G$ is countable if and only if possibly, there is a point $\omega$, whose neighbor graph is connected, with all vertices of degree $2$ except exactly one starting node with degree $1$ (amongst the neighbors of $\omega$)---so that these neighbor vertices form an infinite linked chain from the starting node---and furthermore, all other nodes in the graph are adjacent to distinct neighbors of $\omega$.
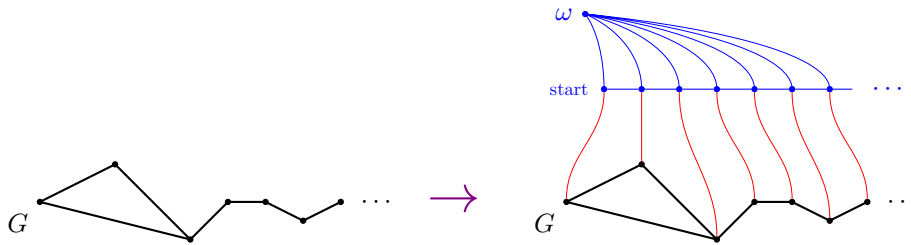
$$\begin{tikzpicture}
  \node[dot,label=below left:$G$] (d) at (0,.5) {};
  \node[dot] (e) at (1,1) {};
  \node[dot] (f) at (2,0) {};
  \node[dot] (g) at (2.5,.5) {};
  \node[dot] (h) at (3,.5) {};
  \node[dot] (i) at (3.5,.25) {};
  \node[dot] (j) at (4,.5) {};
  \node at (4.5,.5) {$\cdots$};
  \draw[thick] (f.center) -- (d.center) -- (e.center) -- (f.center) -- (g.center) -- (h.center) -- (i.center) -- (j.center);
  \draw (5.5,.5) node[scale=2,violet] {$\to$};
\begin{scope}[xshift=7cm]
  \node[dot,label=below left:$G$] (d) at (0,.5) {};
  \node[dot] (e) at (1,1) {};
  \node[dot] (f) at (2,0) {};
  \node[dot] (g) at (2.5,.5) {};
  \node[dot] (h) at (3,.5) {};
  \node[dot] (i) at (3.5,.25) {};
  \node[dot] (j) at (4,.5) {};
  \node at (4.5,.5) {$\cdots$};
  \draw[thick] (f.center) -- (d.center) -- (e.center) -- (f.center) -- (g.center) -- (h.center) -- (i.center) -- (j.center);
  \begin{scope}[red]
  \node[blue,dot,label={[blue]left:$\omega$}] (omega) at (.25,3) {};
  \foreach \k in {0,...,6} {\node[blue,dot] (\k) at (.5+.5*\k,2) {};}
  \draw[very thin,blue] \foreach \n in {.5,...,6.5} {(omega) arc (60:0:\n cm and 1.155cm)};
  \draw[blue,every label/.style={scale=.7}] (0) node[label={[blue]left:start}] {} -- (1) -- (2) -- (3) -- (4) -- (5) -- (6) -- ++(.3,0) +(.5,0) node {$\cdots$};
  \draw[very thin,out=90,in=-90] (d) to (0)
                             (e) to (1)
                             (f) to (2)
                             (g) to (3)
                             (h) to (4)
                             (i) to (5)
                             (j) to (6);
  \end{scope}
\end{scope}
\end{tikzpicture}$$
Note that the ``possibly'' clause includes the finite graphs.
\end{proof}

One can express ``countably infinite'' by saying that the graph is countable but not finite. And uncountability is expressible simply as ``not countable.'' Let us extend this by showing also that having size at most continuum is expressible.

\begin{theorem}\label{Theorem.Size-continuum-is-expressible}
In the class of graphs, the property of having size at most continuum $\continuum$ is expressible in the modal language of graph theory. There is a sentence $\psi_{\smallleq\continuum}\in\mathcal{L}_\sim^{\possible}$ such that for any graph $G$,
 $$G\satisfies\psi_{\smallleq\continuum}\qquad\text{ if and only if }\qquad |G|\leq\continuum.$$
\end{theorem}

\begin{proof}
A graph $G$ has size at most continuum if and only if we can associate every node in the graph with a distinct subset of $\omega$. This is expressible in the modal language of graph theory as follows. The sentence $\psi_{\smallleq\continuum}$ asserts that possibly, there is a node $\omega$ whose neighbors form a connected subgraph in which every node has degree $2$ except one initial node of degree $1$ (within that neighbor graph), such that all other nodes are adjacent to distinct subsets of the neighbors of $\omega$. In other words, for any two distinct nodes $x$ and $y$ not among $\omega$ and its neighbors, there is a neighbor $n$ of $\omega$ such that $x$ is adjacent to $n$ if and only if $y$ is not adjacent to $n$, as indicated in the figure here:
$$
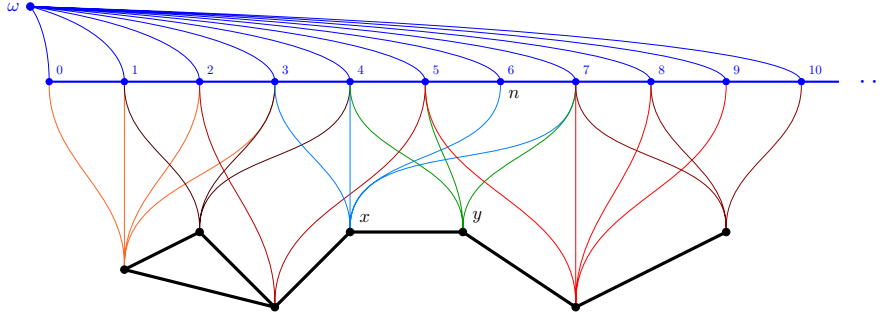
\begin{tikzpicture}[every label/.style={scale=.7}]
  \node[dot=1pt] (A) at (1,.5) {};
  \node[dot=1pt] (B) at (2,1) {};
  \node[dot=1pt] (C) at (3,0) {};
  \node[dot=1pt,label=above right:$x$] (D) at (4,1) {};
  \node[dot=1pt,label=above right:$y$] (E) at (5.5,1) {};
  \node[dot=1pt] (F) at (7,0) {};
  \node[dot=1pt] (G) at (9,1) {};
  \node[dot=1pt,blue,label={[blue]left:$\omega$}] (omega) at (-.25,4) {};
  \draw[very thick] (C.center) -- (A.center) -- (B.center) -- (C.center) -- (D.center) -- (E.center) -- (F.center) -- (G.center);
 \draw[Blue,thick] (0,3) node[dot,label={[scale=.7]above right:$0$}] (0) {}
  \foreach \n in {1,...,10} {-- ++(1,0) node[dot,label={[scale=.7]above right:$\n$}] (\n) {}} -- ++(.5,0) +(.5,0) node {$\cdots$};
 \draw[label distance=-6pt] (6) node[label=below right:$n$] {};
\begin{pgfonlayer}{background}
  \draw[blue,very thin] \foreach \n in {.25,...,10.25} {(omega) arc (60:0:2*\n cm and 1*1.155cm)};
  \foreach \anc/\clr/\nn in {A/Orange/{0,1,2,3}, B/Sepia/{1,3,4}, C/BrickRed/{2,5}, D/RoyalBlue/{3,4,6,7}, E/OliveGreen/{4,5,7}, F/Red/{5,7,8,9}, G/RawSienna/{7,8,10}}
  {\draw[\clr,very thin] \foreach \n in \nn {(\anc) to[out=90,in=-90] (\n)};}
\end{pgfonlayer}
\end{tikzpicture}$$
In this way, every node of the original graph $G$ is associated with a distinct set of the neighbors of $\omega$, and so $G$ will have size at most continuum.
\end{proof}

\begin{theorem}
In the class of graphs, the property of having size exactly continuum $\continuum$ is expressible in the modal language of graph theory. There is a sentence $\psi_\continuum\in\mathcal{L}_\sim^{\possible}$ such that for any graph $G$,
 $$G\satisfies\psi_\continuum\qquad\text{ if and only if }\qquad |G|=\continuum.$$
\end{theorem}

\begin{proof}
This is a little subtler than one might expect---the difficulty is that when one uses the possibility operator, new vertices can be added to the original graph, and there is no way of telling which points are new and which are original. (This is the main point of the actuality operator @ discussed in section \ref{Section.Actuality}.) Nevertheless, we can express the property as follows. A graph $G$ has size continuum if and only if it has size at most continuum and necessarily, if the graph consists of a node $A$ and its neighbors (as in red below), then necessarily, in any extension having two connected components (the other being a node $B$ and its neighbors, as in brownish red), such that the union has size at most continuum, then necessarily, in any extension in which that is exhibited by an association as above of nodes with subsets of $\omega$, then possibly, in a further extension in which that remains true, there is another copy of $\omega$ and a new association of the neighbors of $A$ with distinct subsets of it, in such a way that every pattern for that copy of $\omega$ is realized by a node adjacent to $A$.
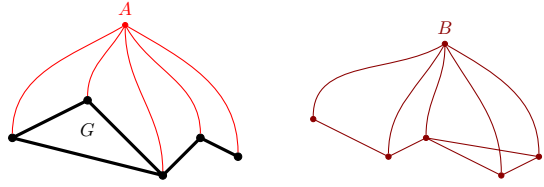
$$\begin{tikzpicture}[every label/.style={scale=.7}]
  \node[dot=1pt] (A) at (1,.5) {};
  \node[dot=1pt,label={[label distance=1.5ex]below:$G$}] (B) at (2,1) {};
  \node[dot=1pt] (C) at (3,0) {};
  \node[dot=1pt] (D) at (3.5,.5) {};
  \node[dot=1pt] (E) at (4,.25) {};
  \draw[very thick] (C.center) -- (A.center) -- (B.center) -- (C.center) -- (D.center) -- (E.center);
  \draw[red] (2.5,2) node[dot,label=above:$A$] (AA) {}
       \foreach \pp [count=\i from 0] in {A,B,C,D,E}{ (AA) to[out=210+30*\i,in=90] (\pp)};
 \draw[RawSienna]
  (5,.75) node[dot] (a) {}
  ++(1,-.5) node[dot] (b) {}
  ++(.5,.25) node[dot] (c) {}
  ++(1,-.5) node[dot] (d) {}
  ++(.5,.25) node[dot] (e) {}
  (a.center) -- (b.center) -- (c.center) -- (d.center) -- (e.center) -- (c.center)
  (c) +(.25,1.25) node[dot,label=above:$B$] (BB) {}
  \foreach \pp [count=\i from 0] in {a,b,c,d,e}{ (BB) to[thin,out=210+30*\i,in=90] (\pp)};
\end{tikzpicture}$$
The point is that if $G$ does have size continuum, then we can find the second desired bijection with the subsets of $\omega$. And if $G$ does not have size continuum, then we may extend $G$ by adding a node $A$ adjacent to every node in $G$ and then add a new node $B$ along with continuum many neighbors, associating these with subsets of $\omega$ in such a way so as to use up all the possible subsets---this will prevent $A$ from getting any new neighbors in an extension in which the neighbors of $A$ and $B$ are still associated with distinct subsets, since all the patterns are used up. And since there are fewer than continuum many neighbors of $A$, they cannot make an association with subsets of a second copy of $\omega$ in such a way that every pattern is realized.
\end{proof}

We have made a good start in this section on illustrating the expressive power of modal graph theory, but we shall actually extend these results much further in section \ref{Section.Interpretative-power-of-modal-graph-theory}, showing that having size $\aleph_\omega$ or $\beth_\omega$ is expressible, as is having size the first $\beth$-fixed point or the next or indeed the first $\beth$-hyperfixed-point and much more.

\section{Some elementary modal model theory}\label{Section.Elementary-modal-model-theory}

Let us now begin to develop some of the more general elementary theory of modal model theory, working in $\Mod(T)$ for an arbitrary first-order theory $T$, beginning with the fact that $\mathcal{L}$-isomorphisms preserve $\mathcal{L}^{\possible}$ truths.

\newtheorem{renaminglemma}[theorem]{Renaming lemma}%
\begin{renaminglemma}
For any first-order theory $T$, isomorphisms of models of $T$ as $\mathcal{L}$-structures preserve $\mathcal{L}^{\possible}$ truth in $\Mod(T)$. That is, if $\pi:M\cong N$ is an isomorphism of models of $T$, then for any modal assertion $\varphi\in\mathcal{L}^{\possible}$, we have
 $$M\satisfies\varphi[a]\qquad\Iff\qquad N\satisfies\varphi[\pi(a)].$$
\end{renaminglemma}

\begin{proof}
The claim of the lemma is clearly true for assertions $\varphi$ in the base language $\mathcal{L}$. And it is clearly preserved by Boolean combinations and also quantifiers. What remains is to check the preservation by modal operators. But any extension of $M$ can be translated to a corresponding isomorphic extension of $N$, preserving and extending the isomorphism. And so by induction the lemma holds for all assertions~$\varphi$.
\end{proof}

Next, we show that $\mathcal{L}$-elementarity is the same as $\possible\mathcal{L}$-elementarity. Observation \ref{Observation.Key-lemma-fails-for-Lpossible} will show that this isn't generally true for $\mathcal{L}^{\possible}$-elementarity.

\begin{keylemma}In the class of models $\Mod(T)$ of a first-order theory $T$,
$$M\elesub_{\mathcal{L}} N\qquad\text{if and only if}\qquad M\elesub_{\possible\mathcal{L}} N.$$
\end{keylemma}

\begin{proof}
We shall show by induction on formulas $\varphi$ that if $M\elesub_{\mathcal{L}} N$, then
 $$M\satisfies\varphi[a]\quad\Iff\quad N\satisfies\varphi[a]$$
for every $\varphi\in\possible\mathcal{L}$, for all models. This equivalence is immediate when $\varphi$ is an assertion in the base language $\mathcal{L}$, and it is clearly preserved through Boolean combinations. All that remains for $\varphi\in\possible\mathcal{L}$ is preservation through modal operators. The downward preservation of possibility is immediate, since if $N\satisfies\possible\varphi[a]$, then also $M\satisfies\possible\varphi[a]$, because $M\of N$ and so any extension of $N$ also extends $M$. To show the upward preservation of possibility, suppose that $M\satisfies\possible\varphi$. So there is an extension $M\of H$ with $H\satisfies\varphi$. Consider the theory $\Delta(H)\union\Delta_0(N)$ consisting of the elementary diagram of $H$ together with the atomic diagram of $N$. That is, we take every $\mathcal{L}$-assertion $\varphi(a_0,\dots,a_n)$ true in $H$, using constants for every element of $H$, together with the atomic and negated atomic truths $\varphi_0(b_0,\dots,b_m)$ of $N$, again using constants for every element of $N$. This theory is finitely consistent, we claim, using the fact that $M\elesub N$. To see this, observe that if $\varphi(a_0,\dots,a_n)$ is true in $H$ and $\varphi_0(b_0,\dots,b_m)$ is true in $N$, where $\varphi_0$ is a conjunction of atomic and negated atomic assertions, then $M\satisfies\exists\bar x\,\varphi_0(\bar x)$ and so we may interpret the constants $b_i$ inside $M$ and hence inside $H$, thereby realizing the conjunction $\varphi(\bar a)\wedge\varphi_0(\bar b)$ in $H$. So the theory is consistent, and therefore we have a model $H'\satisfies\Delta(H)\union\Delta_0(N)$. It follows that $H'$ is an elementary extension of an embedded copy of $H$, and a direct extension of a copy of $N$. We may assume $H\elesub_{\mathcal{L}} H'$, and so by the induction hypothesis we know $H'\satisfies\varphi[a]$. Since $N$ embeds into $H'$ while fixing $a$, we therefore also have $N\satisfies\possible\varphi[a]$.
\end{proof}

We can use the key lemma to prove a version of it for mere elementary equivalence, as follows:

\begin{theorem}\label{Theorem.L-theory-determines-possibleL-theory}
In the class of models $\Mod(T)$ of a first-order theory $T$ in language $\mathcal{L}$, the $\mathcal{L}$-theory of a model determines its $\possible\mathcal{L}$-theory:
$$M\eleequiv_{\mathcal{L}} N\qquad\text{if and only if}\qquad M\eleequiv_{\possible\mathcal{L}} N.$$
\end{theorem}

\begin{proof}
Suppose that $M\eleequiv_{\mathcal{L}} N$ for models of $T$. It follows that the two models have a common elementary extension, a model $M^+$ into which both $M$ and $N$ embed $\mathcal{L}$-elementarily. By the key lemma and the renaming lemma, both of these embeddings are also $\possible\mathcal{L}$-elementary, and so the $\possible\mathcal{L}$-theories of $M$ and $N$ must be the same.
\end{proof}

\begin{corollary}\label{Corollary.Complete-theory}
If $T$ is a complete first-order theory in language $\mathcal{L}$, then in $\Mod(T)$ every $\possible\mathcal{L}$ sentence $\sigma$ is either true in all models of $T$ or none. In this case every true $\possible\mathcal{L}$ sentence $\sigma$ is necessarily true: $\sigma\iff\necessary\sigma$.
\end{corollary}

\begin{proof}
If $T$ is complete, then all models in $\Mod(T)$ are $\mathcal{L}$-elementarily equivalent, and so by theorem \ref{Theorem.L-theory-determines-possibleL-theory} they are also $\possible\mathcal{L}$-elementarily equivalent. So any $\possible\mathcal{L}$ sentence either holds in all models of $T$ or none.
\end{proof}

We can prove a version of theorem \ref{Theorem.L-theory-determines-possibleL-theory} for types.

\begin{theorem}\label{Theorem.L-type-determines-possibleL-type}
In the class of models $\Mod(T)$ of a first-order theory $T$ in language $\mathcal{L}$, the $\mathcal{L}$-type of individuals $\bar a=(a_0,\dots,a_n)$ in a model $M$ determines the $\possible\mathcal{L}$-type of those individuals in that model.
\end{theorem}

\begin{proof}
Suppose that $\bar a$ and $\bar b$ have the same $\mathcal{L}$-type in a model $M\satisfies T$. We may find an elementary extension $M\elesub \overbar M$ for which $\overbar M$ is sufficiently homogeneous, so that some automorphism of $\overbar M$ moves $\bar a$ to $\bar b$. By the key lemma, the $\possible\mathcal{L}$-type of $\bar a$ in $M$ is the same as in $\overbar M$, which by the renaming lemma is the same as the $\possible\mathcal{L}$-type of $\bar b$ in $\overbar M$, which by the key lemma again is the same as the $\possible\mathcal{L}$-type of $\bar b$ in $M$.
\end{proof}

Since theorem \ref{Theorem.L-theory-determines-possibleL-theory} shows that the $\possible\mathcal{L}$ theory of a model is determined completely by its $\mathcal{L}$-theory, one might take this to suggest that $\possible\mathcal{L}$ assertions are actually $\mathcal{L}$-expressible. But this isn't true in general, since we have already seen in theorem \ref{Theorem.2-colorability-is-expressible} that $\possible\mathcal{L}$ can express $2$-colorability of graphs, whereas this concept is not expressible in the base language. Nevertheless, the $2$-colorability of a graph is a consequence of the \emph{theory} of the graph, since it is equivalent to the nonexistence of odd-length cycles. In the general case what we get is the following:

\begin{theorem}\label{Theorem.Infinitary-disjunct-of-infinitary-conjuncts}
In the class $\Mod(T)$ of all models of a first-order theory $T$ in language $\mathcal{L}$, every $\possible\mathcal{L}$ formula $\varphi(x)$ is equivalent to an infinitary disjunction of infinitary conjunctions of $\mathcal{L}$-assertions.
\end{theorem}

\begin{proof}
Let us first consider the case that $\varphi$ is a $\possible\mathcal{L}$ sentence. Consider all the various models $M\satisfies T$ in which $\varphi$ holds. The theorem shows that this depends only on the $\mathcal{L}$-theory of $M$, since any other model with that same theory must also satisfy $\varphi$. That is, $\varphi$ holds in a model $N$ if and only if the $\mathcal{L}$-theory of $N$ is one of the theories of a model in which $\varphi$ holds. Let $\mathcal{T}$ be the set of all the $\mathcal{L}$-theories $\overbar T$ that are the $\mathcal{L}$-theories of such a model $M$, one in which $T+\varphi$ holds. By theorem \ref{Theorem.L-theory-determines-possibleL-theory}, the sentence $\varphi$ is equivalent to the disjunction of these theories $\overbar T$, like this:
 $$\varphi\quad\Iff\quad\bigvee_{\overbar T\in\mathcal{T}}\,\bigwedge_{\psi\in \overbar T}\psi.$$
This is an infinitary disjunction of infinitary conjunctions of $\mathcal{L}$ sentences.

Now suppose that $\varphi(x)$ is a formula, with free variables. Let $\mathcal{L}^+$ be the language where $x$ is considered a constant symbol rather than a variable symbol. Every model of $T$ in $\mathcal{L}$ has diverse expansions to $\mathcal{L}^+$ simply by interpreting $x$ by an element. Evaluating a formula with a free variable $x$ under a valuation of that variable amounts to the same thing as interpreting the corresponding sentence in a model in which $x$ is regarded as a constant symbol. So the models of $\Mod(T)$ in $\mathcal{L}^+$ are the same thing as models in the language $\mathcal{L}$ under a valuation assigning $x$ to an individual of the model. But in $\Mod(T)$ considered in the language $\mathcal{L}^+$, we may apply the analysis of the previous paragraph to realize $\varphi(x)$ as an infinitary disjunction of infinitary conjunctions of $\mathcal{L}^+$-assertions. So $\varphi(x)$ is equivalent to a disjunction of conjunctions of formulas of $\mathcal{L}$.
\end{proof}

Perhaps part of what is going on in theorem \ref{Theorem.Infinitary-disjunct-of-infinitary-conjuncts} is the following:

\begin{theorem}\label{Theorem.Expansions}
If $T$ is a theory in a first-order language $\mathcal{L}$ having only nonempty models and $\mathcal{L}^+$ is an expansion of $\mathcal{L}$, then for any modal assertion $\varphi$ in $\mathcal{L}^{\possible}$, a model $M$ of $T$ satisfies $\varphi(a)$ in $\Mod_{\mathcal{L}}(T)$ if and only if every expansion $M^+$ of it also does so in $\Mod_{\mathcal{L}^+}(T)$.
\end{theorem}

In this theorem, we have two different versions of the potentialist system $\Mod(T)$, depending on whether you consider $\mathcal{L}$-structures or $\mathcal{L}^+$-structures. The point is that the two systems have the same modal truths for assertions in the unexpanded modal language $\mathcal{L}^{\possible}$.

\begin{proof}
Every $\mathcal{L}$-model $M\satisfies T$ has expansions to the language $\mathcal{L}^+$, and what we claim is that if $M^+$ is an expansion of $M$ to the language $\mathcal{L}^+$, then
 $$M\satisfies\varphi(a)\quad\Iff\quad M^+\satisfies\varphi(a).$$
This we prove by induction on $\varphi$. The claim is clear for $\mathcal{L}$ assertions, since $\varphi$ does not mention any of the new vocabulary of $\mathcal{L}^+$. And it is preserved by Boolean combinations and quantifiers, since the domains of $M$ and $M^+$ are the same. What remains is the modal operator. If $M\satisfies\possible\varphi(a)$, then there is an extension $M\of N$ with $N\satisfies\varphi(a)$. We may find an expansion of $N$ to a model $N^+$ extending $M^+$, simply by interpreting the extra structure first on $M$ as in $M^+$, and then interpreting it on the rest of $N$. By the induction hypothesis, $N^+\satisfies\varphi(a)$, and so $M^+\satisfies\possible\varphi(a)$. Conversely, if $M^+\satisfies\possible\varphi(a)$, then there is some extension $N^+$ of $M^+$ satisfying $\varphi(a)$. The reduct $N$ of $N^+$ to the $\mathcal{L}$ language will therefore satisfy $\varphi(a)$ by the induction hypothesis, and this is an extension of $M$. So $M\satisfies\possible\varphi(a)$.
\end{proof}

In a countable language $\mathcal{L}$, there are at most continuum many theories $\overbar T$ completing $T$, and so theorem \ref{Theorem.Infinitary-disjunct-of-infinitary-conjuncts} provides an at-most-continuum size disjunction of countable conjunctions of $\mathcal{L}$ sentences. By considering the negation, we would also get a representation of $\varphi$ as a size-continuum conjunction of countable disjunctions of $\mathcal{L}$-sentences. Can these infinitary disjunctions and conjunctions be reduced to countable? We suspect not in general.

\begin{question}
In the class $\Mod(T)$ for a first-order theory $T$ in a countable language $\mathcal{L}$, is every $\possible\mathcal{L}$ assertion equivalent to an assertion of $\mathcal{L}_{\omega_1,\omega}$?
\end{question}

But there are some important cases where the answer is actually positive, in a very strong way. Recall that a theory $T$ admits \emph{quantifier elimination} with respect to a language if every assertion in that language is equivalent in $\Mod(T)$ to a quantifier-free formula. Similarly, a theory $T$ admits \emph{modality elimination} with respect to a language, if every assertion in that language is equivalent in $\Mod(T)$ to a modality-free assertion. A stronger notion would be the \emph{modality trivialization}, which holds over a language if $\possible\varphi(x)$ is equivalent specifically to $\varphi(x)$ in $\Mod(T)$ for every $\varphi$ in that language.

\begin{theorem}
If a theory $T$ admits quantifier elimination with respect to its language $\mathcal{L}$, then it admits simultaneous quantifier and modality elimination---every assertion in $\mathcal{L}^{\possible}$ is equivalent in $\Mod(T)$ to a Boolean combination of atomic formulas. Furthermore, the theory admits modality trivialization over $\mathcal{L}^{\possible}$, so that $\possible\varphi(x)$ is equivalent to $\varphi(x)$ for any assertion $\varphi\in\mathcal{L}^{\possible}$.
\end{theorem}

\begin{proof}
Let us prove by induction on formulas that every formula $\varphi(x)$ in the full modal language $\mathcal{L}^{\possible}$ is equivalent in $\Mod(T)$ to a quantifier/modality-free assertion. Since $T$ admits quantifier elimination, we know this is true for the $\mathcal{L}$ formulas. And this property is preserved by Boolean connectives, as well as by quantifiers, which can subsequently be eliminated. So suppose that $\varphi(x)$ is quantifier/modality-free and consider $\possible\varphi(x)$. This holds in a model $M$ if and only if there is an extension $M\of N\satisfies T$ with $N\satisfies\varphi[a]$. But since $\varphi$ is quantifier/modality-free, this will be absolute back to the original model, so $M\satisfies\varphi[a]$, as desired.

It now follows generally that $\possible\varphi(x)$ is equivalent to $\varphi(x)$, since if $\psi(x)$ is the quantifier/modality-free formula equivalent to $\varphi(x)$, then $\possible\varphi(x)$ is equivalent to $\possible\psi(x)$, which is equivalent to $\psi(x)$ since $\psi$ is quantifier/modality-free, and this is equivalent to $\varphi(x)$. So $\possible\varphi(x)$ is equivalent to $\varphi(x)$ for any $\varphi(x)$ in $\mathcal{L}^{\possible}$.
\end{proof}

A theory $T$ is \emph{model complete} if every instance of the submodel relation $M\of N$ for models of $T$ is an elementary submodel $M\elesub N$. This is a weakening of quantifier eliminability, for it is equivalent to the property that every formula in $\mathcal{L}$ is equivalent in models of $T$ to a universal assertion \cite[theorem~8.3.1]{Hodges1993:ModelTheory}.

\begin{theorem}\label{Theorem.Modality-trivialization-equivalents}
For any first-order theory $T$, the following are equivalent:
 \begin{enumerate}
   \item $T$ admits modality trivialization over all assertions in $\mathcal{L}^{\possible}$.
   \item $T$ admits modality trivialization over all assertions in $\possible\mathcal{L}$.
   \item $T$ admits modality trivialization over all assertions in $\mathcal{L}$.
   \item $T$ is model complete.
 \end{enumerate}
\end{theorem}

\begin{proof}
$(1\to 2\to 3)$ Immediate.

$(3\to 4)$ Assume that modalities trivialize in $\Mod(T)$ over assertions in $\mathcal{L}$. Suppose that $M\of N$ are models of $T$. If $N\satisfies\varphi(a)$ for some $a\in M$, then $M\satisfies\possible\varphi(a)$, which by modality trivivialization means that $M\satisfies\varphi(a)$, and so $M\elesub N$. So the theory is model complete.

$(4\to 1)$. Assume that the theory $T$ is model complete. We shall show by induction on formulas that $\possible\varphi(x)$ is equivalent in $\Mod(T)$ to $\varphi(x)$ for every assertion $\varphi\in\mathcal{L}^{\possible}$. Suppose that the claim is true for $\varphi$ and all subformulas of $\varphi$. It follows that $\varphi(x)$ is equivalent to the modality-free assertion $\varphi^*(x)$, obtained from $\varphi$ by deleting all its modal operators. Thus, $\varphi^*$ is an assertion of $\mathcal{L}$. If $M\satisfies\possible\varphi(a)$, then there is some extension $M\of N$ with $N\satisfies\varphi(a)$, which by our assumption means that $N\satisfies\varphi^*(a)$. By the model completeness of $T$, it follows that $M\satisfies\varphi^*(a)$, and consequently also $M\satisfies\varphi(a)$. So the modality has trivialized, as desired.
\end{proof}

The following example shows that modality elimination is a strictly weaker property than modality trivialization.

\begin{observation}
 There is a theory $T$ that is not model complete and consequently does not admit modality trivialization, but it does admit modality elimination---every assertion in $\mathcal{L}^{\possible}$ is equivalent in $\Mod(T)$ to an assertion in $\mathcal{L}$.
\end{observation}

\begin{proof}
Let $T$ be the theory of a dense linear order having a least and greatest element. This is not model complete, because one can have a suborder $M\of N$ which changes the particular element that is least or greatest, and so this will not be an elementary submodel. So by theorem \ref{Theorem.Modality-trivialization-equivalents} the theory does not admit modality trivialization. But we claim that the theory does admit modality elimination. Indeed, we claim that every assertion in $\mathcal{L}^{\possible}$ is equivalent to a Boolean combination of atomic assertions and assertions that particular variables are least or greatest. These kinds of assertions are closed under Boolean combinations, and one can handle the quantifier case just as in the elimination of quantifiers argument for endless dense linear orders. Finally, consider $\possible\varphi(x_0,\ldots,x_n)$, where $\varphi$ is such a Boolean combination. By putting $\varphi$ into disjunctive normal form and then distributing the $\possible$ over the disjunction, we may assume that $\varphi$ is a conjunction of atomic or negated atomic assertions and assertions that particular variables $x_i$ are or are not least or greatest. But true atomic facts are necessarily true; an element $x$ is possibly least if and only if it is actually least; it is possibly greatest if and only if it is actually greatest; and every $x$ is possibly not least and possibly not greatest. These observations allow us to express $\possible\varphi$ without the modal operator.
\end{proof}

Theorem \ref{Theorem.Modality-trivialization-equivalents} shows that modality trivialization over $\mathcal{L}$ is equivalent to trivialization over $\possible\mathcal{L}$ and $\mathcal{L}^{\possible}$. The corresponding fact for modality elimination is the following:

\begin{theorem}
Every first-order theory $T$ admits modality elimination for assertions in $\possible\mathcal{L}$ if and only if it does so for $\mathcal{L}^{\possible}$.
\end{theorem}

\begin{proof}
The converse implication is immediate. For the forward implication, suppose that $T$ admits modality elimination for $\possible\mathcal{L}$-assertions. We now prove inductively that every assertion $\varphi(x)$ of $\mathcal{L}^{\possible}$ is equivalent in $\Mod(T)$ to an $\mathcal{L}$-assertion $\psi(x)$. This is true for the $\mathcal{L}$-assertions themselves, and it is preserved by Boolean combinations and quantifiers. Also, it is preserved by the modal operator $\possible$, precisely because $\possible\psi(x)$ is an assertion of $\possible\mathcal{L}$, if $\psi\in\mathcal{L}$, and we have assumed that $T$ eliminates modalities for $\possible\mathcal{L}$ assertions.
\end{proof}

The \emph{\Lowenheim\ number} of a language with respect to a class of models is the smallest cardinal $\kappa$ such that any assertion $\varphi$ of that language that is satisfiable in a model of the class is satisfiable in such a model of size at most $\kappa$.

\begin{theorem}\label{Theorem.Lowenheim-Skolem}
For any first-order theory $T$, the upward and downward \Lowenheim-Skolem theorems hold in $\Mod(T)$ with respect to $\possible\mathcal{L}$, using the same cardinals as for $\mathcal{L}$. That is, every model $M\satisfies T$ has $\possible\mathcal{L}$-elementarity extensions and submodels of all the same cardinalities as it does for $\mathcal{L}$-elementarity. Consequently the \Lowenheim\ number of $\possible\mathcal{L}$ with respect to $\Mod(T)$ is the same as for $\mathcal{L}$ with respect to $\Mod(T)$.
\end{theorem}

\begin{proof}
This is immediate from the key lemma and the classical \Lowenheim-Skolem theorem, since the key lemma shows that $\mathcal{L}$-elementarity is the same as $\possible\mathcal{L}$-elementarity. So any $\mathcal{L}$-elementary substructure or extension is also $\possible\mathcal{L}$-elementary.
\end{proof}

%
%
%
%
%

Let us now begin to observe that many of the properties we have proved for the intermediate modal language $\possible\mathcal{L}$ do not extend to the full modal language $\mathcal{L}^{\possible}$. The main lesson seems to be that $\possible\mathcal{L}$ is in many respects closer to $\mathcal{L}$ than it is to $\mathcal{L}^{\possible}$. Nevertheless, some classical model theoretic principles fail even for $\possible\mathcal{L}$.

\begin{observation}\label{Observation.Key-lemma-fails-for-Lpossible}
The equivalence stated in the key lemma does not hold generally for elementarity in the full modal language $\mathcal{L}^{\possible}$. In modal graph theory, there can be elementary substructures that are not $\mathcal{L}^{\possible}$-elementary and not even $\mathcal{L}^{\possible}$-elementarily equivalent.
\end{observation}

\begin{proof}
Consider the class of all graphs. By the \Lowenheim-Skolem theorem, we can have $M\elesub_{\mathcal{L}}N$, where $M$ is a countable graph and $N$ is uncountable. But by theorem \ref{Theorem.Countability-is-expressible}, countability is $\mathcal{L}_\sim^{\possible}$ expressible for graphs,
showing the failure of $M\elesub_{\mathcal{L}_\sim^{\possible}}N$ and indeed of $M\equiv_{\mathcal{L}_\sim^{\possible}}N$.
\end{proof}

The same example establishes the following:

\begin{observation}
The equivalence stated in theorem \ref{Theorem.L-theory-determines-possibleL-theory} does not generally extend to equivalence in the full modal language $\mathcal{L}^{\possible}$. In modal graph theory, there can be elementarily equivalent models that are not $\mathcal{L}^{\possible}$-elementarily equivalent.
\end{observation}


We are not sure whether corollary \ref{Corollary.Complete-theory} extends to the full modal language $\mathcal{L}^{\possible}$.

\begin{question}
Is there a complete first-order theory $T$ for which not all models in $\Mod(T)$ have the same $\mathcal{L}^{\possible}$ theory?
\end{question}

\noindent One candidate may be to take $T$ as the theory of true arithmetic, and then argue somehow that in a nonstandard model of $T$ one can define the standard cut in the modal language and be thereby enabled to make assertions about the standard system of the model in $\mathcal{L}^{\possible}$. But we have not yet been able to make this idea work.

\begin{observation}
Theorem \ref{Theorem.L-type-determines-possibleL-type} does not hold generally for types in the full modal language $\mathcal{L}^{\possible}$---in modal graph theory, the $\mathcal{L}$-type of an individual in a model of $T$ does not necessarily determine the $\mathcal{L}^{\possible}$ type of that individual in $\Mod(T)$.
\end{observation}

\begin{proof}
Consider the class of graphs. Let $G$ be a graph consisting of two stars, one with center vertex $a$ having countably infinitely many neighbors and one with center vertex $b$ having uncountably many neighbors, and no other vertices or edges.
$$
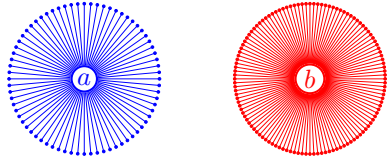
\begin{tikzpicture}
\draw[blue] (0,0) node[circle,thick,draw,inner sep=1pt] (a) {$a$};
\draw[blue,very thin] \foreach \v in {0,5,...,355} {(a) -- +(\v:1) node[dot=.2] {} };
\draw[red] (3,0) node[circle,thick,draw,inner sep=1pt] (b) {$b$};
\draw[red,very thin] \foreach \v in {0,3,...,359} {(b) -- +(\v:1) node[dot=.2] {} };
\end{tikzpicture}$$
The types of $a$ and $b$ are the same in the language of graph theory, but they are not the same in the language of modal graph theory, which can express the fact that $a$ has only countably many neighbors, while $b$ has uncountably many.
\end{proof}

\begin{observation}
The statement of theorem \ref{Theorem.Infinitary-disjunct-of-infinitary-conjuncts} does not generally hold for assertions in the full modal language $\mathcal{L}^{\possible}$. In modal graph theory there are $\mathcal{L}^{\possible}$-expressible assertions that are not equivalent to any infinitary Boolean combination of $\mathcal{L}$ assertions.
\end{observation}

\begin{proof}
Theorem \ref{Theorem.Countability-is-expressible} shows that countability is expressible in modal graph theory, but the assertion that the graph is uncountable cannot be equivalent to any infinitary Boolean combination of sentences in the language of graph theory, since the truth of such assertsion is preserved from an uncountable model to a countable elementary substructure.
\end{proof}

\begin{observation}
The conclusion of theorem \ref{Theorem.Lowenheim-Skolem} does not hold generally for the full modal language $\mathcal{L}^{\possible}$. In modal graph theory, no uncountable graph has a countable $\mathcal{L}^{\possible}$-elementary substructure, and the \Lowenheim\ number for modal graph theory in the language $\mathcal{L}_\sim^{\possible}$, if it exists, is enormous.
\end{observation}

\begin{proof}
The first part of this is an immediate consequence of the fact that countability is expressible in $\mathcal{L}^{\possible}$, as shown in theorem \ref{Theorem.Countability-is-expressible}. Every countable submodel will satisfy the countability assertion, but the whole uncountable graph will not.

The \Lowenheim\ number of modal graph theory must be at least as large as every cardinal $\kappa$ whose size is expressible in modal graph theory. By the results of section \ref{Section.Expressive-power-of-modal-graph-theory}, these cardinals reach to the continuum, and in section \ref{Section.Interpretative-power-of-modal-graph-theory} we shall prove that they exceed the first $\beth$-fixed point and indeed the first $\beth$-hyper-fixed point.
\end{proof}

\begin{observation}
The compactness property can fail for $\possible\mathcal{L}$-theories (and hence also for $\mathcal{L}^{\possible}$ theories). Specifically, there is a theory $T$ in the language of modal graph theory $\possible\mathcal{L}_\sim$ that is finitely realized in the class of graphs but not fully realized.
\end{observation}

\begin{proof}
Let $T$ assert that the graph is not $2$-colorable, but that there are no cycles of length $1$, none of length $2$, none of length $3$, and so on---there are no cycles of any particular finite length $n$. This theory is expressible in $\possible\mathcal{L}_\sim$, since theorem \ref{Theorem.2-colorability-is-expressible} shows that $2$-colorability is expressible in $\possible\mathcal{L}_\sim$, and the assertion that there is no $n$-cycle is expressible in $\mathcal{L}_\sim$ as a separate assertion for each $n$.

This theory is finitely satisfiable, because if only finitely many sizes of cycles are ruled out, we can satisfy the assertion with a large enough odd cycle, which will not be $2$-colorable. The theory altogether, however, is not satisfiable, since every cycle-free graph is $2$-colorable. One can color every node by the parity of the length of the shortest path to a given fixed node in each connected component.
\end{proof}

One can also naturally construct a violation of compactness using the fact that finiteness is expressible, since one can consider the theory asserting that the graph is finite, but has size at least $n$ for every particular finite $n$. This theory would indeed be finitely realizable and not realizable in the class of graphs. Since finiteness is expressible only in $\mathcal{L}_\sim^{\possible}$, however, this way of arguing would get the violation of compactness only in this stronger language, rather than in the weaker fragment $\possible\mathcal{L}$ as we have above.

Although we have presented these examples as violations of compactness, there is another sense in which we have not violated compactness here. Namely, to satisfy an assertion of the modal language does not mean just to provide a model $M$ in which it is true in the potentialist system $\Mod(T)$, but rather, to provide a whole new potentialist system, a class of models with an accessibility relation, not necessarily $\Mod(T)$ under direct extensions. The observation above, in contrast, is about realizing a theory in a graph of $\Mod(T)$, but always working in this same Kripke model. For this reason, one might look upon the counterexamples as akin to violations of saturation in this Kripke model, rather than violations of compactness for the modal logic.

\begin{observation}
The \Los\ theorem for ultraproducts can fail for assertions in the modal language $\possible\mathcal{L}$. Specifically, in modal graph theory there is a statement $\varphi\in\possible\mathcal{L}_\sim$ that holds in some graphs $G_n$, but fails in every nonprincipal ultraproduct $\prod_n G_u/U$.
\end{observation}

\begin{proof}
Let the graphs $G_n$ be odd-length cycles of increasing size. None of these graphs is $2$-colorable, but the ultraproduct consists of uncountably many disconnected $\Z$-chains, and this is $2$-colorable.
\end{proof}

\section{A quibble about the accessibility relation}

We should like to dispense with a certain issue about the accessibility relation in $\Mod(T)$, for there are actually two natural accessibility relations for $\Mod(T)$. On the one hand, it is natural to define possibility as we have above, as \emph{direct extension possibility}, using the submodel or direct extension relation $M\of N$, defining that $M\satisfies\possible\varphi[a]$ if and only if there is a model $N$ with $M\of N$ and $N\satisfies\varphi[a]$. This accessibility relation is very natural from the perspective of potentialism, by which one views each of the various models as a fragment approximating the larger universe to which they are building. On this account, individuals come into actual existence in a world and then persist through all subsequent larger worlds.

On the other hand, it is also natural to define a variant of this possibility notion using embeddability $M\ofsim N$ in place of direct extensions $M\of N$. Namely, the \emph{embedded extension possibility} operator is defined by: $M\satisfies\embedpossible\varphi[a]$ if and only if there is an embedded extension $M\ofsim N$, that is, a model $N$ with an embedding $j:M\to N$, for which $N\satisfies\varphi[j(a)]$. Thus, every individual in $M$ has a counterpart in $N$ via the embedding. Notice that there may be more than one embedding $j$ of $M$ into $N$, and so the models $M$ and $N$ alone do not necessarily determine the counterpart correspondence of their individuals, but rather the particular embedding $j:M\to N$.

As an accessibility relation, the embedded extension relation $M\ofsim N$ exhibits far better algebraic closure properties than direct extension $M\of N$. For example, in nontrivial cases the direct extension relation $M\of N$ is not directed, and neither does it exhibit amalgamation or convergence (see definitions in section \ref{Section.Modal-validities}), since two different extensions might happen to use the same individual object in incompatible ways. But the embedded extension relation often has all of these features, since the particular identity of individuals no longer matters when they are embedded into another common structure. From this perspective, the embedded extension relation is often more robust algebraically.

In terms of modal assertions, we are glad to say that in light of theorem \ref{Theorem.Direct-extension-possibility-same-as-embedded-extension} the entire issue is moot, because although as accessibility relations the notions are different, nevertheless it turns out that the corresponding modal operators coincide. In this sense, we can equivalently think of possibility with respect to either notion, and the result will be the same.

\begin{theorem}\label{Theorem.Direct-extension-possibility-same-as-embedded-extension}
In $\Mod(T)$ for any first-order theory, direct extension possibility and embedded extension possibility are equivalent for assertions $\varphi$ in the full modal language $\mathcal{L}^{\possible}$ with parameters from $M$.
$$M\satisfies\possible\varphi[a]\qquad\Iff\qquad M\satisfies\embedpossible\varphi[a]$$
Similarly, direct extension necessity and embedded extension necessity are also equivalent.
$$M\satisfies\necessary\varphi[a]\qquad\Iff\qquad M\satisfies\embednecessary\varphi[a]$$
\end{theorem}

\begin{proof}
We aim to prove that $M\satisfies\possible\varphi[a]$ if and only if $M\satisfies\embedpossible\varphi[a]$. The forward direction is immediate, since direct extensions $M\of N$ are instances of embedded extensions $j:M\to N$ using the inclusion embedding $j(x)=x$. Conversely, if $M\satisfies\embedpossible\varphi[a]$, then there is some embedded extension $j:M\to N$ with $N\satisfies\varphi[j(a)]$. The model $M$ is isomorphic via $j$ with the range of the embedding $M'=j\image M$, and this model directly extends to $N$. So $M'\satisfies\possible\varphi[j(a)]$, and therefore $M\satisfies\possible\varphi[a]$ by the renaming lemma, as desired. The necessity case now follows by duality.
\end{proof}

It follows inductively that complex formulas involving $\embedpossible$ and $\embednecessary$ have the same meaning as the corresponding assertions using $\possible$ and $\necessary$.

Going beyond this, Sam Adam-Day has proved (in forthcoming work) that the two potentialist systems are actually bisimilar, which completely explains why they exhibit the same modal truths.

\begin{theorem}[Adam-Day]
For any first-order theory $T$, the potentialist systems consisting of $\Mod(T)$ under the direct extension relation or under the embedded extension relation, are bisimilar. Therefore, they exhibit exactly the same modal truths.
\end{theorem}

\section{Modal validities}\label{Section.Modal-validities}

One of the central research tasks of modal model theory is to discover the modal principles and validities that hold in a given potentialist system and the models in it. This has been the focus of prior work on set-theoretic potentialism \cite{Hamkins2003:MaximalityPrinciple}, \cite{HamkinsLoewe2008:TheModalLogicOfForcing, HamkinsLoewe2013:MovingUpAndDownInTheGenericMultiverse}, \cite{HamkinsLeibmanLoewe2015:StructuralConnectionsForcingClassAndItsModalLogic},
\cite{HamkinsLinnebo:Modal-logic-of-set-theoretic-potentialism},
\cite{HamkinsWilliams:The-universal-finite-sequence},
\cite{HamkinsWoodin:The-universal-finite-set}, and arithmetic potentialism \cite{Hamkins:The-modal-logic-of-arithmetic-potentialism}. In modal model theory generally, we should like to move beyond set theory and arithmetic to the models of any given theory, to graphs, groups, fields, orders, and what have you. We aim to discover which modal principles are valid in $\Mod(T)$ for a first-order theory $T$, considering this as a potentialist system. A key definition is the following:

\begin{definition}\label{Definition.Valid}\upshape
A modal assertion $\varphi(p_1,\ldots,p_n)$, with propositional variables $p_i$, is \emph{valid} at a world $M$ in a potential system $\mathcal{W}$ for an allowed language of instances, if all substitution instances $\varphi(\psi_1,\ldots,\psi_n)$ arising for $\psi_i$ in that language are true at $M$ in $\mathcal{W}$.
\end{definition}

Thus, a modal validity $\varphi(p_0,\ldots,p_n)$ stands as a template scheme for all the modal truths $\varphi(\psi_0,\ldots,\psi_n)$ that arise by substituting the propositional variables $p_i$ with allowed substitution assertions $\psi_i$. Whether a modal assertion is valid or not is often highly sensitive to the particular language of allowed substitutions $\psi$, for instance, whether we are allowing $\psi$ only from the language $\mathcal{L}$, or from $\possible\mathcal{L}$ or $\mathcal{L}^{\possible}$, or whether parameters from the model $M$ are allowed. When speaking of a modal validity, therefore, one must take care to clarify precisely the class of substitution instances for which validity is asserted.

For this reason, we should particularly like to emphasize that when we say a modal theory such as \theoryf{S4.2} is valid in a model $M$ of a potentialist system, we are treating \theoryf{S4.2} as a propositional modal theory, that is, as a set of assertions in propositional modal logic $\mathcal{P}$, rather than as a logic, as a proof system with rules of inference. In general, because the validities are sensitive to the precise class of substitution instances and the allowed parameters, they do not always interact well with the inference rules typically used when defining a modal logic, if those rules would be taken to be applicable generally in the broader predicate-logic modal context with variables and quantifiers. For example, a modal principle $\varphi(a)$ can be valid at a model $M$ using any parameter $a$ from that model, but not valid in some extension $M\of N$ using a parameter from the extension $N$ (see remarks after theorem \ref{Theorem.MP-for-graphs} for a specific example). In this case, $\forall x\, \varphi(x)$ will be valid at $M$, but not $\necessary\forall x\,\varphi(x)$, and therefore the validities at $M$ are not closed under necessitation, even when the modal theory \theoryf{S5} is valid there in the sense of definition \ref{Definition.Valid}.

Let us begin by reviewing some easy lower bounds. Consider the potentialist system consisting of the class $\Mod(T)$ of all models of a fixed first-order theory $T$. Before stating the theorem, we recall some useful terminology, particularly on the distinction between convergence and amalgamation. The class of models $\Mod(T)$ exhibits embedding \emph{convergence}, if for any model $M$ and embeddings $M\ofsim N_0$ and $M\ofsim N_1$, then there is a model $N$ with $N_0\ofsim N$ and $N_1\ofsim N$. Convergence is thus a form of local directedness, since it asserts that models $N_0$ and $N_1$ have a common embedding extension, provided that these models are themselves embedding extensions of a common model $M$. Note that for convergence there is no requirement that the diagram of embeddings commutes. This is the key difference between convergence and amalgamation and is why convergence leads only to validities for sentences rather than also for assertions with parameters. Namely, the class of models has \emph{amalgamation}, if whenever $j_0:M\to N_0$ and $j_1:M\to N_1$ are embeddings, there there are embeddings $h_0:N_0\to N$ and $h_1:N_1\to N$ to a common model $N$, such that the diagram commutes, meaning $h_0\circ j_0=h_1\circ j_1$.\goodbreak

\begin{theorem}\label{Theorem.Validities-lower-bounds}\
 \begin{enumerate}
   \item The modal theory \theoryf{S4} is valid at every model in $\Mod(T)$ with respect to all substitution instances, as is every instance of the converse Barcan formula: $$\necessary\forall x\,\varphi(x)\implies\forall x\,\necessary\varphi(x)$$
   \item If the class of models of $T$ is convergent under embeddability $\ofsim$, then \theoryf{S4.2} is valid at every model in $\Mod(T)$ for $\mathcal{L}^{\possible}$-sentences.
   \item If the class of models of $T$ exhibits amalgamation, then \theoryf{S4.2} is valid at every model $M$ in $\Mod(T)$ for $\mathcal{L}^{\possible}$-assertions with parameters from $M$.
   \item If the models of $T$ are linearly pre-ordered by embeddability $\ofsim$, then \theoryf{S4.3} is valid at every model $M$ in $\Mod(T)$ for $\mathcal{L}^{\possible}$-assertions with parameters from $M$.
 \end{enumerate}
\end{theorem}

Note that the theorem is concerned with the convergence and amalgamation properties using the embedding extension relation $\ofsim$, rather than direct extension $\of$, and this is important, since $\of$ is almost never convergent, while in important cases the embedding extension relation $\ofsim$ exhibits both convergence and amalgamation.

\begin{proof}[Proof hints]
We leave the proof details as an exercise in modal reasoning for the reader. For the first part of statement (1), use that the direct extension relation $\of$ is reflexive and transitive; for the converse Barcan scheme, use that the domains are inflationary with respect to the accessibility relation. For statement (2), the main fact is that convergence is sufficient to establish the validity of $\possible\necessary\varphi\to\necessary\possible\varphi$ for any sentence $\varphi$, using theorem \ref{Theorem.Direct-extension-possibility-same-as-embedded-extension} to transfer from $\ofsim$ to $\of$. If there are parameters involved, however, as in statement (3), then one should use amalgamation to verify $\possible\necessary\varphi(a)\to\necessary\possible\varphi(a)$; the commutativity of the diagram enables one to know that it is the same $a$ being referred to around both sides of the diagram. For statement (4), linearity is sufficient to verify the validity of the \theoryf{S4.3} axiom $(\possible\varphi\wedge\possible\psi)\to[\possible(\varphi\wedge\possible\psi)\vee\possible(\psi\wedge\possible\varphi)]$.
\end{proof}

Let us turn now to the more difficult issue of providing upper bounds on the validities of a potentialist system. A general method has emerged in a series of papers \cite{HamkinsLoewe2008:TheModalLogicOfForcing}, \cite{HamkinsLeibmanLoewe2015:StructuralConnectionsForcingClassAndItsModalLogic},
\cite{HamkinsLinnebo:Modal-logic-of-set-theoretic-potentialism}, \cite{HamkinsWilliams:The-universal-finite-sequence}, \cite{HamkinsWoodin:The-universal-finite-set}, \cite{Hamkins:The-modal-logic-of-arithmetic-potentialism}, developing the \emph{control statement} technique of establishing upper bounds on the modal validities of a potentialist system. With this method, as in theorem \ref{Theorem.Control-statement-upper-bounds}, one uses the existence of various kinds of control statements in the potentialist system---buttons, switches, dials, ratchets, or railyards---to establish upper bounds on the class of modal validities of the system.

The main advantage of the control statement method is that it enables one to analyze and discover the modal validities of a potentialist system by using principally only expertise in the subject matter of the object theory, rather than technical expertise in the foundations of modal logic. To construct control statements for potentialist conceptions in set theory, group theory, or graph theory generally requires only set-theoretic, group-theoretic, or graph-theoretic ideas, respectively. With those ideas and with theorem \ref{Theorem.Control-statement-upper-bounds}, one can often establish important facts about the modal validities of one's potentialist conception.

Let us quickly review some of the key kinds of control statements. A \emph{button} is a statement $\varphi$ that is possibly necessary, that is, for which $\possible\necessary\varphi$; it is \emph{pushed} when $\necessary\varphi$ holds, otherwise unpushed. A \emph{switch} is a statement $\psi$ for which $\necessary(\possible\psi\wedge\possible\neg\psi)$. A \emph{dial} is a sequence of statements $d_1$, \ldots, $d_n$, such that necessarily, exactly one of the statements is true, and any of them is possible. A \emph{ratchet} is a sequence of button statements $r_1$, \ldots, $r_n$ such that each is possibly necessary, each necessarily implies the previous, and $\bigwedge_{k<n}\possible(\neg r_{k+1}\wedge\necessary r_k)$, so that each can be pushed without pushing the next. A \emph{railway switch} is a statement $r$ for which both $\necessary r$ and $\necessary \neg r$ are possible. A \emph{railyard} is an assignment of statements to a finite tree pre-order, such that the possibility relations of the statements are exactly in accordance with the tree order. A family of control statements are \emph{independent}, if necessarily they can be operated without interference with one another. For example, a family of buttons and switches are independent, if necessarily, one can push exactly any desired button (and no others) while not affecting the switches and also one can set the switches as desired without pushing any buttons.

As we had mentioned, the significance of the existence of these kinds of control statements in a potentialist system is that they provide natural upper bounds on the collection of validities, as summarized in the following theorem:

\begin{theorem}\label{Theorem.Control-statement-upper-bounds}%
Assume $\mathcal{W}$ is a potentialist system.\nobreak
  \begin{enumerate}
    \item If a world $M$ in $\mathcal{W}$ has arbitrarily large finite families of independent switches, then the modal validities at $M$ are contained within S5.
    \item If a world $M$ in $\mathcal{W}$ has arbitrarily large finite families of independent buttons and switches, or independent buttons and dials, then the modal validities are contained within \theoryf{S4.2}.
    \item If a world $M$ in $\mathcal{W}$ has arbitrarily long ratchets, independent of switches and dials (as many as desired), then the modal validities are contained within \theoryf{S4.3}.
    \item If a world $M$ in $\mathcal{W}$ admits railyard labelings of every finite pre-tree, then the modal validities at $M$ are exactly \theoryf{S4}.
  \end{enumerate}
In each case, the relevant language of substitution instances would be any language containing the control statements and closed under Boolean connectives.
\end{theorem}

For proof of theorem \ref{Theorem.Control-statement-upper-bounds}, we refer the reader to \cite{HamkinsLeibmanLoewe2015:StructuralConnectionsForcingClassAndItsModalLogic} for statements (1), (2), and (3), and to
\cite{Hamkins:The-modal-logic-of-arithmetic-potentialism} for statement (4).

Let us illustrate these ideas in the case of modal graph theory.

\begin{theorem}\label{Theorem.Graph-validities}
 In graphs,
 $$\theoryf{S4.2}\of\text{ modal validities of graphs }\of \theoryf{S5}.$$
Indeed, for validities with respect to substitution instances in the language of graph theory with parameters, every graph validates either exactly \theoryf{S4.2} or exactly \theoryf{S5}, and both cases are realized. Specifically:
\begin{enumerate}
  \item The modal theory \theoryf{S4.2} is valid in any graph for any assertion at all, with parameters.
  \item The validities of any graph, with respect even just to substitution instances for sentences in the language of graph theory, is contained within \theoryf{S5}.
  \item Some graphs have their validities exactly \theoryf{S4.2}, with respect even just to sentences in the language of graph theory.
  \item Some graphs have \theoryf{S5} as valid, even for assertions in the full modal language of graph theory $\mathcal{L}_\sim^{\possible}$, with parameters.
  \item For validities with respect to substitution instances in the language of graph theory with parameters, every graph validates either exactly \theoryf{S4.2} or exactly \theoryf{S5}.
\end{enumerate}
\end{theorem}

\begin{proof}
For the \theoryf{S4.2} lower bound, it suffices by theorem \ref{Theorem.Validities-lower-bounds} to observe that the class of all graphs exhibits the amalgamation property. And indeed it does, since if graph $G$ embeds into graphs $H_0$ and $H_1$, then we may take copies of $H_0$ and $H_1$ extending $G$ and for which the sets of respective new vertices are disjoint. We may then form a graph $H$ with embedded copies of $H_0$ and $H_1$, which therefore amalgamate over $G$ as desired.

For the \theoryf{S5} upper bound, it suffices to show that every graph admits arbitrarily large finite dials. Let $d_n$ be the graph-theoretic assertion that there are exactly $n$ isolated points, and let $d_{\geq N}$ assert that there are at least $N$ isolated points. These are assertable by sentences in the first-order language of graph theory. For any finite number $N$, we may consider the statements $d_n$ for $n<N$ and $d_{\geq N}$, which together form a dial, since every graph either has some specific number of isolated points $n<N$, or else it has at least $N$ isolated points; and every graph can be extended to a graph making any one of these dial statements true.

To show that the lower bound of \theoryf{S4.2} is realized, it suffices to find a graph admitting arbitrarily large finite families of independent buttons and dials. We can use the dials $d_n$ mentioned just previously, and for buttons, let $b_k$ (for $k\geq 3$) assert that there is a cycle of length $k$ in the graph. This is a button, since every graph can be extended to a graph with a $k$-cycle, and once there is a $k$-cycle, then it remains in all larger extension graphs. Furthermore, these buttons and dials are independent, since we may add any number of isolated points without changing the existence of $k$-cycles, and we may add any $k$-cycle without changing the number of isolated points. If a graph lacks $k$-cycles for infinitely many $k$, then these buttons will be unpushed, and consequently the validities of the graph will be contained within \theoryf{S4.2} and hence identical to \theoryf{S4.2}, as claimed for statement (3).

Statement (4) will follow from theorem \ref{Theorem.Unions-of-chains-S5} in the next section from the observation that the class of graphs is closed under unions of chains. Theorems \ref{Theorem.MP-for-graphs} and \ref{Theorem.MP-for-graphs-no-parameters} identify exactly the graphs that validate \theoryf{S5}.

Statement (5) will similarly be proved in the next section with corollary \ref{Corollary.Graph-validities-S4.2-or-S5}.
\end{proof}

\section{Maximality principle}

A model $M$ in a potentialist system satisfies the \emph{maximality principle} with respect to a language when it satisfies all instances of the \axiomf{S5} axiom $\possible\necessary\varphi\to \varphi$ for $\varphi$ in that language. Our goal in this section is to determine necessary and sufficient conditions for a model $M$ to satisfy the maximality principle with respect to various languages.

Let us begin by establishing in very general circumstances that there will be some models validating \theoryf{S5}.\pagebreak

\begin{theorem}\label{Theorem.Unions-of-chains-S5}%
Suppose that $\mathcal{W}$ is a potentialist system.
\begin{enumerate}
  \item If every countable chain of extensions in $\mathcal{W}$ has an upper bound in $\mathcal{W}$, then every model can be extended to one that validates $\theoryf{S5}$ with respect to any fixed countable family of substitution instances.
  \item If every (set-sized) chain of extensions in $\mathcal{W}$ has an upper bound in $\mathcal{W}$, then every model can be extended to one that validates $\theoryf{S5}$ with respect to sentences in any fixed set-sized language.
  \item If every (set-sized) chain of extensions in $\mathcal{W}$ has its union in $\mathcal{W}$, then every model can be extended to one that validates \theoryf{S5} with respect to any fixed set-sized language, allowing parameters.
\end{enumerate}
\end{theorem}

\begin{proof}
For statement (1), assume that the potentialist system has every countable chain bounded. Fix the assertions $\varphi_0$, $\varphi_1$,\ldots that will be used for substitution, and consider any model $M_0$. If this model thinks $\possible\necessary\varphi_0$, then there is an extension $M_0\of M_1$ with $M_1\satisfies\necessary\varphi_0$. Thus, we have pushed this button. Now, push each next button in turn: if $M_n\satisfies\possible\necessary\varphi_n$, then extend to $M_n\of M_{n+1}\satisfies\necessary\varphi_n$. Thus, we build a chain of models.
 $$M_0\of M_1\of M_2\of M_3\of \cdots $$
By assumption, there is a model $N$ extending every model $M_n$ in the chain. The key observation is that $N\satisfies\possible\necessary\varphi_n\to \varphi_n$, since if $\varphi_n$ is possibly necessary over $N$, then it was possibly necessary at stage $n$ and therefore pushed earlier, making $\varphi_n$ necessary in $N$, as desired.

For statement (2), it is a similar idea, but with a longer chain. Enumerate the sentences as $\varphi_\alpha$ for $\alpha<\delta$, and assume that every chain in $\mathcal{W}$ of size at most $\delta$ has an upper bound. Start in any model $M_0$, and then push button $\varphi_0$, if possible, by finding an extension $M_1\satisfies\necessary\varphi_0$. Continue to $M_1$, $M_2$, pushing each button in turn, if possible. At limit stages $\xi$, we use the chain hypothesis to find a model $M_\xi$ extending all earlier $M_\alpha$ for $\alpha<\xi$. In this way, we build a chain of models $M_\alpha$, for $\alpha<\delta$, such that $\varphi_\alpha$ is pushed in $M_{\alpha+1}$, if it is possibly necesssary over $M_\alpha$.
 $$M_0\of M_1\of M_2\of \cdots \of M_\alpha\of\cdots$$
Using the chain hypothesis a final time, there is a model $M_\delta$ extending the chain. Any statement $\varphi_\alpha$ that is possibly necessary over $M_\delta$ was possibly necessary over $M_\alpha$, and therefore made necessary in $M_{\alpha+1}$ and hence is still true in $M_\delta$. So $M_\delta$ validates all instances of the maximality principle for these sentences.

Statement (3) is a bit more subtle, since we are allowing parameters. Because the models are growing in size, we cannot seem to predict in advance how many buttons we will need to push. Nevertheless, we will be able to catch our tail using the continuity assumption and an iteration of iterations. Assume that the potentialist system is closed under unions of chains. Validating the \axiomf{S5} axiom $\possible\necessary p\to p$ amounts essentially to a closure operation, and so there will be a model validating every required instance. Start in any model $M_0$, and then using the method of statement (2), we may push all the buttons $\varphi(a)$ that arise over $M_0$ using parameters $a\in M_0$ for statements $\varphi$ in any set-sized language. Thus, we'll find a model $M_1$ satisfying all instances of the \axiomf{S5} axiom using parameters in $M_0$. We can similarly find a further extension $M_2$ satisfying the instances of \axiomf{S5} using parameters in $M_1$. Continuing in this way, consider the sequence of extensions
 $$M_0\of M_1\of M_2\of\cdots\of M_n\of\cdots$$
whose limit $M_\omega=\Union_n M_n$ is a model in $\mathcal{W}$ by our closure assumption, and it will satisfy all instances of \axiomf{S5} using any parameters in $M_\omega$, as desired.
\end{proof}

In the case of $\Mod(T)$ for a first-order theory $T$, this result plays out as follows.

\begin{corollary}\label{Corollary.S5-in-AEtheories}\
\begin{enumerate}
  \item For any first-order theory $T$, every model of $T$ can be extended to one in which \theoryf{S5} is valid in $\Mod(T)$ for all sentences in $\mathcal{L}^{\possible}$ or indeed any set-sized language.
  \item If the theory is $\forall\exists$ axiomatizable, then every model of it can be extended to one in which \theoryf{S5} is valid in $\Mod(T)$ for all assertions in $\mathcal{L}^{\possible}$ or indeed any set-sized language, with parameters.
\end{enumerate}
\end{corollary}

\begin{proof}
A simple compactness argument shows that every chain $\<M_\alpha\mid\alpha<\theta>$ of models of $T$ has an upper bound in $\Mod(T)$. Simply consider the theory consisting of $T$ together with the atomic diagrams $\Union_{\alpha<\theta}\Delta_0(M_\alpha)$. This theory is finitely consistent, and hence consistent, and any model of it provides an upper bound for the chain. Therefore, by theorem \ref{Theorem.Unions-of-chains-S5}, every model of $T$ can be extended to a model in which \theoryf{S5} for any set of sentences in any set-sized language.

For statement (2), if the theory is $\forall\exists$ axiomatizable, then unions of chains of models of $T$ will be models of $T$, placing us into the continuity case of theorem \ref{Theorem.Unions-of-chains-S5}, statement (3). In this case, therefore, we can find a limit model that will validate \theoryf{S5} for any fixed set-sized language, allowing parameters from the model.
\end{proof}

One cannot in general omit the requirement of the $\forall\exists$ axiomatization, in light of the following observation, which provides an $\exists\forall$ theory $T$ having no models that validate \theoryf{S5} for $\mathcal{L}$ with parameters.

\begin{observation}\label{Observation.Theory-without-S5-models}
There is an $\exists\forall$ axiomatizable first-order theory $T$ for which $\Mod(T)$ admits no models with the maximality principle with parameters.
\end{observation}

\begin{proof}
In the language of graph theory, let $T$ assert that ``there is a node adjacent to all other nodes,'' and consider the resulting potentialist system $\Mod(T)$ of graphs having this property. No graph like this can satisfy the maximality principle with parameters, since it is possibly necessary that any given node $a$ lacks an edge to some other node, since we can extend the graph by adding a node to which $a$ is not adjacent, even while adding a new universally adjacent node in order to satisfy~$T$.
\end{proof}

A similar example would be the theory of dense linear orders with a least element---it is a button to make any particular element non-least. Both of these theories have $\exists\forall$ axiomatizations, and it seems that this argument will work with any truly $\exists\forall$ theory. Such theories are not closed under unions of chains, and this is how they escape the conclusions of statement (3) in theorem \ref{Theorem.Unions-of-chains-S5} and statement (2) of corollary \ref{Corollary.S5-in-AEtheories}.

We would now like to understand more precisely which models validate \theoryf{S5}, and to do so, let us begin with modal graph theory, which illustrates the central ideas in a concrete manner. The \emph{countable random graph} is (up to isomorphism) the unique countable saturated graph. This graph is characterized by the \emph{finite pattern property}, which asserts that for any two disjoint finite sets of vertices $A$ and $B$, there is a vertex $v$ adjacent to every vertex in $A$ and to no vertex in $B$. The finite pattern property allows you to find a new node satisfying exactly any desired pattern of connectivity to finitely many previous nodes, and by iterating this in a back-and-forth manner, one can easily show both that any two countable graphs with this property are isomorphic and also that any graph with this property is universal for all countable graphs.

\begin{theorem}\label{Theorem.MP-for-countable-graphs}
 A countable graph $G$ satisfies the maximality principle
  $$\possible\necessary\varphi(\bar a)\to\varphi(\bar a)$$
 for all $\varphi$ in the language of graph theory with parameters $\bar a\in G$ if and only if $G$ is the countable random graph.
\end{theorem}

\begin{proof}
($\to$) Suppose that $G$ satisfies the maximality principle assertion with parameters. Every instance of the finite pattern property has the form $$\exists x\ (\bigwedge_i x\sim a_i)\wedge(\bigwedge_j x\nsim b_j),$$
with finitely many parameters $a_i,b_j\in G$. Every such statement is possibly necessary, since we can extend the graph to add such a vertex, if necessary, and once there is one, it remains in all further extensions. So by the maximality principle, this statement must already be true in $G$. And so $G$ fulfills the finite pattern property, and since it is countable, it is therefore the countable random graph.

($\leftarrow$) Suppose that $G$ is the countable random graph. We aim to show that $G$ satisfies the maximality principle. To show this, suppose that $G\satisfies\possible\necessary\varphi(\bar a)$ for some $\bar a=(a_0,\ldots,a_n)$ from $G$. So there is an extension graph $G\of H$ such that $H\satisfies\necessary\varphi(\bar a)$. By theorem \ref{Theorem.Lowenheim-Skolem}, we may assume that $H$ is countable. Therefore $H$ embeds into $G$ by the universality of the countable random graph, and furthermore, we can find such an embedding that fixes the parameters in $\bar a$. So $G$ extends the image of $H$ under that embedding, and so $G\satisfies\varphi(\bar a)$, as desired.
\end{proof}

We can mount several refinements of this theorem. The argument actually shows that the countable random graph exhibits the maximality principle for assertions in the modal language $\possible\mathcal{L}_\sim$, not just $\mathcal{L}_\sim$. And conversely, having the maximality principle merely for existential assertions in the language of graph theory is sufficient to ensure the finite pattern property; so these all will be equivalent. In addition, we didn't use countability as such, but rather merely the finite pattern property (which together with countability characterizes the countable random graph), and so a more erudite version of the theorem is the following:

\begin{theorem}\label{Theorem.MP-for-graphs}
The following are equivalent for any graph $G$.
\begin{enumerate}
  \item The graph $G$ satisfies the maximality principle $\possible\necessary\varphi(\bar a)\to\varphi(\bar a)$ for every assertion $\varphi$ in the modal language of graph theory $\possible\mathcal{L}_\sim$, allowing parameters $\bar a=(a_0,\ldots,a_n)$ from $G$.
  \item The graph $G$ satisfies the maximality principle $\possible\necessary\varphi(\bar a)\to\varphi(\bar a)$ for every assertion $\varphi$ in the language of graph theory $\mathcal{L}_\sim$, allowing parameters $\bar a=(a_0,\ldots,a_n)$ from $G$.
  \item The graph $G$ satisfies the maximality principle $\possible\necessary\varphi(\bar a)\to\varphi(\bar a)$ for every existential assertion $\varphi=\exists x\,\varphi_0(x,\bar a)$ in the language of graph theory, with $\varphi_0$ quantifier free, allowing parameters $\bar a=(a_0,\ldots,a_n)$ from $G$.
  \item The graph $G$ fulfills the finite pattern property.
\end{enumerate}
\end{theorem}

\begin{proof}
($1\to 2\to 3$) These implications are immediate.

($3\to 4$) Every instance of the finite pattern property is a possibly necessary existential statement, which must therefore already be true in the graph.

($4\to 1$) This is the main part of the argument; we argue as in theorem \ref{Theorem.MP-for-countable-graphs}. Suppose that $G$ fulfills the finite pattern property and satisfies $\possible\necessary\varphi(\bar a)$ for some $\varphi\in\possible\mathcal{L}_\sim$, with parameters $\bar a$ from $G$. So there is a graph extension $G\of H\satisfies\necessary\varphi(\bar a)$. The theory $\Delta(G)\union\Delta_0(H)$ is consistent, since any finite part of $\Delta_0(H)$ can be realized inside $G$ over the parameters $\bar a$ by successive uses of the finite pattern property. Any model $N$ of this theory will have $G\elesub N$ and $H\of N$. Because of the latter, we will have $N\satisfies\varphi(\bar a)$, and therefore by the key lemma, we will deduce $G\satisfies\varphi(\bar a)$, since $\varphi$ is in $\possible\mathcal{L}_\sim$.
\end{proof}

Theorem \ref{Theorem.MP-for-graphs} shows that a model can validate \theoryf{S5} for assertions in $\possible\mathcal{L}$ with parameters, while this fails in an extension (using parameters of the extension). For example, the countable random graph certainly fulfills the finite pattern property, but if we were to add a new isolated vertex, we would of course no longer fulfill the finite pattern property.

\begin{corollary}\label{Corollary.Graph-validities-S4.2-or-S5}
For any graph $G$, the validities of $G$ with respect to the language of graph theory with parameters is either exactly \theoryf{S4.2} or exactly \theoryf{S5}.
\end{corollary}

\begin{proof}
This corollary amounts to theorem \ref{Theorem.Graph-validities} statement (5). If $G$ fulfills the finite pattern property, then the validities of $G$ are exactly \theoryf{S5} by theorem \ref{Theorem.Graph-validities}. So suppose that $G$ does not fulfill the finite pattern property. So there are finitely many $a_0,\ldots,a_n$ and $b_0,\ldots,b_m$ for which there is no vertex $x$ connected to every $a_i$ and to no $b_j$. This is an unpushed button. For any other vertex $u$ not amongst the $a_i$s and $b_j$s, we may consider the statement asserting that there is a vertex adjacent to every $a_i$ and also to $u$, and to none of the $b_j$. This also is a button, which is unpushed in $G$, and these buttons are independent, since one can push any one of them without pushing any of the others. And furthermore, these buttons are independent of the dials asserting that there are exactly $k$ isolated points or at least $N$ isolated points, as in theorem \ref{Theorem.Graph-validities}. So we've found arbitrarily large independent families of buttons and dials, and so by theorem \ref{Theorem.Control-statement-upper-bounds} the validities of $G$ are contained within \theoryf{S4.2}. But since \theoryf{S4.2} is valid in any graph, the validities of $G$ are exactly \theoryf{S4.2}, as desired.
\end{proof}

Let us now also prove a parameter-free version of theorem \ref{Theorem.MP-for-graphs}.\goodbreak

\begin{theorem}\label{Theorem.MP-for-graphs-no-parameters}
The following are equivalent for any graph $G$.
\begin{enumerate}
  \item The graph $G$ satisfies the maximality principle $\possible\necessary\varphi\to\varphi$ for all sentences $\varphi$ in the modal language of graph theory $\possible\mathcal{L}_\sim$.
  \item The graph $G$ satisfies the maximality principle $\possible\necessary\varphi\to\varphi$ for all sentences $\varphi$ in the language of graph theory $\mathcal{L}_\sim$.
  \item The graph $G$ satisfies the maximality principle $\possible\necessary\varphi\to\varphi$ for all existential sentences $\varphi=\exists x_0,\ldots,x_n\ \varphi_0(x_0,\ldots,x_n)$, with $\varphi_0$ quantifier free in the language of graph theory $\mathcal{L}_\sim$.
  \item The graph $G$ is universal for finite graphs.
\end{enumerate}
\end{theorem}

\begin{proof}
($1\to 2\to 3$) These implications are immediate.

($3\to 4$) The assertion that $G$ contains an isomorphic copy of a given finite graph is a possibly necessary existential statement, since we can add such a copy to any given graph, and so it must already be true in $G$, if $G$ satisfies the maximality principle for such assertions. So any such graph $G$ as in statement (3) is universal for finite graphs.

($4\to 1$) Suppose that $G$ is universal for finite graphs. We want to prove that it fulfills every instance of the maximality principle $\possible\necessary\varphi\to\varphi$ for sentences $\varphi$ in the modal language of graph theory. Assume $G\satisfies\possible\necessary\varphi$. This means that there is an extension $G\of H$ for which $H\satisfies\varphi$. Consider the theory $\Th(G)\union\Delta_0(H)$, taking the elementary theory of $G$ in the language of graph theory (sentences only, no parameters) together with the atomic diagram of $H$ (using names for every element of $H$). This theory is finitely consistent, since any finite part of $H$ can be realized inside $G$ by finite universality. So the theory is satisfied in some graph, which because of $\Th(G)$ will satisfy the full theory of $G$ and because of $\Delta_0(H)$ will contain a copy of $H$. So there is a model $N$ of this theory, which we may take as a direct extension $H\of N$. Since $N$ extends $H$, it follows that $N\satisfies\varphi$. And since $N$ satisfies the theory of $G$, it follows from theorem \ref{Theorem.L-theory-determines-possibleL-theory} that $G\satisfies\varphi$. So $G$ satisfies all instances of the sentential maximality principle, as desired.
\end{proof}

We should like now to generalize this analysis to arbitrary theories, not just graph theory. What is really going on? The model-theoretic feature in play here is \emph{existential closure}. A model $M$ is \emph{existentially closed} with respect to a theory $T$, if every existential statement $\exists x_0,\ldots,x_n\,\varphi(x_0,\ldots,x_n,\bar a)$, with $\varphi$ quantifier-free, that is true in some extension $M\of N\satisfies T+\exists \bar x\,\varphi(\bar x,\bar a)$, with parameters $\bar a$ from $M$, is true already in $M$.

\begin{theorem}\label{Theorem.MP-equivalents}
In the potentialist system $\Mod(T)$ for any theory $T$ in a first-order language $\mathcal{L}$, the following are equivalent:
  \begin{enumerate}
    \item $M\satisfies\possible\necessary\varphi(\bar a)\to\varphi(\bar a)$ for formulas $\varphi\in\possible\mathcal{L}$ with parameters $\bar a$ from $M$.
    \item $M\satisfies\possible\necessary\varphi(\bar a)\to\varphi(\bar a)$ for formulas $\varphi\in\mathcal{L}$ with parameters $\bar a$ from $M$.
    \item $M\satisfies\possible\necessary\varphi(\bar a)\to\varphi(\bar a)$ for existential formulas $\varphi=\exists\bar x\varphi_0(\bar x,\bar a)\in\mathcal{L}$, with $\varphi_0$ quantifier free and parameters $\bar a$ from $M$.
    \item $M$ is existentially closed in $\Mod(T)$.
  \end{enumerate}
\end{theorem}\goodbreak

\begin{proof}
($1\to 2\to 3$) These implications are immediate.

($3\to 4$) Any instance of existential closure is a possibly necessary existential statement, since once an existential statement becomes true, it remains true in all further extensions.

($4\to 1$) Suppose that $M$ is existentially closed in $\Mod(T)$. We aim to prove the maximality principle in $M$ for assertions in the modal language $\possible\mathcal{L}$ with parameters. Suppose that $M\satisfies\possible\necessary\varphi(\bar a)$ for such an assertion $\varphi$. So there is an extension $M\of H\satisfies T$ with $H\satisfies\necessary\varphi(\bar a)$. Consider the theory $\Delta(M)\union\Delta_0(H)$, consisting of the full elementary diagram of $M$ together with the atomic diagram of $H$. This theory is finitely consistent, since any finite combination of atomic assertions true in $H$ must already be realizable in $M$, since it is existentially closed. So there is a model $N$ of this theory, which we may take to have $M\elesub N$ and $H\of N$. Since $H\of N$, we know that $N\satisfies\varphi(\bar a)$, and since $M\elesub N$ it follows by the key lemma that $M\satisfies\varphi(\bar a)$, as desired.
\end{proof}

This theorem has a consequence when \emph{every} model in $\Mod(T)$ validates \theoryf{S5}. 

\begin{corollary}\label{Corollary.S5-valid-at-every-model}
For any first-order theory $T$, the following are equivalent:
 \begin{enumerate}
  \item \theoryf{S5} is valid in $\Mod(T)$ at every model, with parameters, for assertions in~$\mathcal{L}^{\possible}$.
  \item \theoryf{S5} is valid in $\Mod(T)$ at every model, with parameters, for assertions in~$\possible\mathcal{L}$.
  \item \theoryf{S5} is valid in $\Mod(T)$ at every model, with parameters, for assertions in~$\mathcal{L}$.
  \item $T$ is model complete.
  \item The modally trivializing logic $p\iff\possible p$ (and so also $p\iff\necessary p$) is valid in $\Mod(T)$ at every model, with parameters, for assertions in $\mathcal{L}^{\possible}$.
 \end{enumerate}
\end{corollary}

\begin{proof}
($1\to 2\to 3$) Immediate.

($3\to 4$) If \theoryf{S5} is valid in $\Mod(T)$ at every model, with parameters, for assertions in $\mathcal{L}$, then by theorem \ref{Theorem.MP-equivalents}, it follows that every model of $T$ is existentially closed. But \cite[theorem~8.3.1]{Hodges1993:ModelTheory} shows that the models of $T$ are all existentially closed if and only if the theory is model complete.

%
($4\to 5$) If the theory $T$ is model complete, then by theorem \ref{Theorem.Modality-trivialization-equivalents}, we have modality trivialization in the full modal language $\mathcal{L}^{\possible}$.

($5\to 1$) If modalities trivialize in $\mathcal{L}^{\possible}$, then $p\iff\possible\necessary p$ is valid at every world, for assertions with parameters. In particular, $\possible\necessary\varphi(a)\to \varphi(a)$ will hold at every world of $\Mod(T)$ for assertions $\varphi$ in $\mathcal{L}^{\possible}$, with parameters.
\end{proof}

\begin{corollary}
For any first-order theory $T$, the modal validities of $\Mod(T)$ for assertions in $\mathcal{L}$, $\possible\mathcal{L}$, or $\mathcal{L}^{\possible}$, with parameters, are never exactly \theoryf{S5}.
\end{corollary}

\begin{proof}
Corollary \ref{Corollary.S5-valid-at-every-model} shows that if \theoryf{S5} is valid in $\Mod(T)$ at every world for any of those languages, with parameters, then $p\iff\possible p$ and $p\iff \necessary p$ are also valid. But these latter validities are not part of \theoryf{S5}, and so the modal validities are not \emph{exactly}~\theoryf{S5}.
\end{proof}\goodbreak

We can similarly establish a parameter-free version, defining that a model $M$ is \emph{sententially existentially closed} in $\Mod(T)$ if whenever $M\of N\satisfies T$ and $N$ satisfies an existential sentence $\exists x_0,\ldots,x_n\,\varphi_0(x_0,\ldots,x_n)$, with $\varphi_0$ quantifier free, then the sentence is already true in $M$. In graph theory, this is equivalent to being universal for finite graphs.

\begin{theorem}\label{Theorem.MP-equivalents-no-parameters}
In the potentialist system $\Mod(T)$ for any theory $T$ in a first-order language $\mathcal{L}$, the following are equivalent:
  \begin{enumerate}
    \item $M\satisfies\possible\necessary\varphi\to\varphi$ for sentences $\varphi\in\possible\mathcal{L}$.
    \item $M\satisfies\possible\necessary\varphi\to\varphi$ for sentences $\varphi\in\mathcal{L}$.
    \item $M\satisfies\possible\necessary\varphi\to\varphi$ for existential sentences $\varphi=\exists x_0,\ldots,x_n\,\varphi_0(x_0,\ldots,x_n)$ with $\varphi_0$ quantifier free in $\mathcal{L}$.
    \item $M$ is sententially existentially closed in $\Mod(T)$.
  \end{enumerate}
\end{theorem}

\begin{proof}
($1\to 2\to 3$) These implications are immediate.

($3\to 4$) Any instance of sentential existential closure is a possibly necessary existential sentence, which by statement (3) must be true in $M$.

($4\to 1$) Suppose that $M$ is sententially existentially closed in $\Mod(T)$. To verify the sentential maximality principle, suppose that $M\satisfies\possible\necessary\varphi$ for some sentence $\varphi\in\possible\mathcal{L}$. So there is an extension $M\of H\satisfies T$ with $H\satisfies\varphi$. Consider the theory $\Th(M)\union\Delta_0(H)$ consisting of the theory of $M$ (no parameters) together with the atomic diagram of $H$, using names for every element of $H$. This theory is finitely consistent, since any finite part of the atomic diagram of $H$ is realized inside $M$ by sentential existential closure. So there is a model $N$ of the theory, which has $M\eleequiv N$ and $H\of N$. Since $H\of N$, we know that $N\satisfies\varphi$, and since $M\eleequiv N$, it follows by theorem \ref{Theorem.L-theory-determines-possibleL-theory} that $M\satisfies\varphi$, as desired.
\end{proof}

Note that if a first-order theory $T$ is complete, then corollary \ref{Corollary.Complete-theory} shows that the modally trivializing logic $\sigma\iff\necessary \sigma$ is valid in $\Mod(T)$ for \emph{sentences} $\sigma$ in $\possible\mathcal{L}$. In particular, the maximality principle $\possible\necessary\sigma\to\sigma$ will also be valid for such sentences.

\begin{observation}
The maximality principle for assertions in the intermediate modal language $\possible\mathcal{L}$ is not necessarily equivalent to the maximality principle in the full modal language $\mathcal{L}^{\possible}$. In the class of graphs, the countable random graph satisfies the maximality principle for $\possible\mathcal{L}_\sim$ with parameters, but it does not satisfy the maximality principle for $\mathcal{L}_\sim^{\possible}$, even merely for sentences.
\end{observation}

\begin{proof}
We have already proved that the countable random graph satisfies the maximality principle in the language $\possible\mathcal{L}_\sim$ with parameters. But this does not extend to the full modal language $\mathcal{L}_\sim$ because the countable random graph is countable, whereas being uncountable is a possibly necessary statement of $\mathcal{L}_\sim^{\possible}$, since every graph can be extended to an uncountable graph, and once a graph is uncountable, this remains true in all further extensions.
\end{proof}

In light of the results of sections \ref{Section.Expressive-power-of-modal-graph-theory} and \ref{Section.Interpretative-power-of-modal-graph-theory}, any graph satisfying the maximality principle in the full modal language $\mathcal{L}^{\possible}_\sim$ must be enormous, because it is a button that the graph exceeds in size any of the cardinalities that we proved to be expressible in this language, such as the first $\beth$-hyper-fixed-point cardinal. Indeed, because of the uncertain status of question \ref{Question.Set-theory-interpreted-in-modal-graph-theory?}, it is not clear whether we can even define in \ZFC\ a truth predicate for modal graph-theoretic truth in the full language; and without such a truth predicate, it is not sensible to assert the maximality principle in that full modal language. For this reason, the maximality principle in the full modal language may be subject to the subtle metamathematical difficulties we discuss in section \ref{Section.Metamathematical-issues}.

\section{The interpretative power of modal graph theory}\label{Section.Interpretative-power-of-modal-graph-theory}

We should like now to demonstrate the strength of modal graph theory by proving that it can interpret various other mathematical theories and structures.

Let us begin with finite modal graph theory, where we consider the potentialist system consisting of all finite graphs. This is a proper class, but if desired one can reduce to a set by considering the finite graphs using a vertex set contained in the set of natural numbers $\N$. By the renaming lemma, any modal truth in the class of all finite graphs will be realized in these kinds of graphs, and so modal truth is not affected by this restriction.

\begin{theorem}\label{Theorem.Arithmetic-is-interpretable}
True arithmetic is interpretable in finite modal graph theory. There is a translation $\varphi\mapsto\varphi^*$ from arithmetic sentences $\varphi$ to sentences of modal graph theory $\varphi^*$, such that $\<\N,+,\cdot,0,1,<>\satisfies\varphi$ if and only if $\varphi^*$ holds in every finite graph in the class of finite graphs.
\end{theorem}

\begin{proof}
The strategy will be to represent numbers within graph theory and then to show that the arithmetic operations and structure are definable in modal graph theory. To begin, we shall represent the numbers with graphs of the following form:
$$
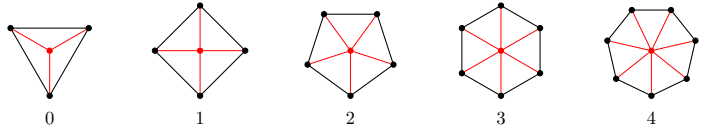
\begin{tikzpicture}
\foreach \n [count=\nn from 0] in {3,...,7} {
\begin{scope}[xshift=2*\n cm,scale=.6]
\pgfmathsetmacro\theta{360/\n}
\draw[rotate=-90] (0:1) \foreach \t [count = \i from 0] in {0,\theta,...,359} { -- (\t:1) node[dot] (P\i) {} } -- cycle;
\draw (0,0) node[dot,red] (\n) {};
\draw[red,very thin] \foreach \t [count = \i from 0] in {0,\theta,...,359} {(\n) -- (P\i) };
\draw (0,-1.5) node[scale=.7] {$\nn$};
\end{scope}
}
\end{tikzpicture}$$
We represent each number $n$ with a vertex (in red) whose neighbors form a cycle of length $n+3$. Our use of this $+3$ offset enables a convenient absoluteness property: if a vertex represents a number in a graph, then in any extension of this graph where it still represents a number, it must represent the same number as originally. (This wouldn't have been true, for example, had we represented $0$ with an isolated node or $2$ with a node adjacent to two adjacent nodes.) So we represent $n$ with an $(n+3)$-cycle. A node encodes a number exactly when it's neighbors form a nonempty connected set with all vertices of degree $2$, and this is expressible in the language of finite modal graph theory.

The arithmetic order relation $n\leq m$ on these representations is expressible in the language of finite modal graph theory. Namely, if we have a representation of $n$ and a representation of $m$, then $n\leq m$ if and only if possibly, there is a node whose neighbors are each adjacent to distinct neighbors of $n$ and $m$, in such a way that these associations form an injection from all but three neighbors of $n$ with at most all but three neighbors of $m$, like this:
$$
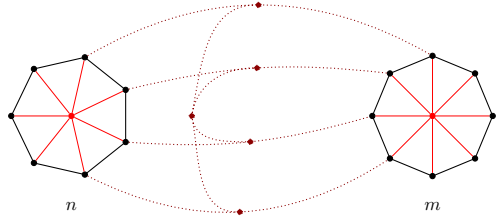
\begin{tikzpicture}[scale=.8]
\pgfmathsetmacro\theta{360/7}
\draw[rotate=-180] (0:1) \foreach \t [count = \i from 0] in {0,\theta,...,359} { -- (\t:1) node[dot] (P\i) {} } -- cycle;
\draw (0,0) node[dot,red] (n) {};
\draw[red,very thin] \foreach \t [count = \i from 0] in {0,\theta,...,359} {(n) -- (P\i) };
\draw (0,-1.5) node[scale=.7] {$n$};
\begin{scope}[xshift=6cm]
\pgfmathsetmacro\theta{360/8}
\draw (0:1) \foreach \t [count = \i from 0] in {0,\theta,...,359} { -- (\t:1) node[dot] (Q\i) {} } -- cycle;
\draw (0,0) node[dot,red] (m) {};
\draw[red,very thin] \foreach \t [count = \i from 0] in {0,\theta,...,359} {(m) -- (Q\i) };
\draw (0,-1.5) node[scale=.7] {$m$};
\end{scope}
\draw[RawSienna,densely dotted] (P2) to[bend right] node[midway,dot] (R0) {} (Q5)
   (P3) to[bend right=10] node[midway,dot] (R1) {} (Q4)
   (P4) to[bend left=10] node[midway,dot] (R2) {} (Q3)
   (P5) to[bend left] node[midway,dot] (R3) {} (Q2)
   (2,0) node[dot] (R) {}
    (R) to[in=-180,out=80] (R3)
    (R) to[in=-180,out=80] (R2)
    (R) to[in=-180,out=-80] (R1)
    (R) to[in=-180,out=-80] (R0);
  \end{tikzpicture}$$
This particular figure illustrates that $4\leq 5$. We can similarly express that $n$ and $m$ represent the same number, by asserting that the association uses all but three neighbors of each vertex.

We can also express that $n+m=r$ for number representatives using the language of finite modal graph theory. We simply say that possibly, they all still represent numbers and there is such a one-to-one correspondence from all but three neighbors of $n$ and all but three neighbors of $m$ with all but three neighbors of $r$. And we can express that $nm=r$, by asserting that possibly, they all still represent numbers and there is a node adjacent to nodes that are each adjacent to a distinct node of $r$, and to neighbors of $n$ and $m$, in such a way that except for three neighbors each of $n$, $m$ and $r$, every pair of nodes from $n$ and $m$ arises exactly once in association with a neighbor of $r$.

Finally, let us explain how to translate arithmetic assertions to finite modal graph theory. Every arithmetic assertion is equivalent to an arithmetic assertion in which there are no compound terms, so that all atomic formulas have the form $x+y=z$, $x\cdot y=z$, $x=y$, $x<y$, $x=0$, or $x=1$, where $x$, $y$ and $z$ are variable symbols. We can translate these into the language of finite modal graph theory by asserting that $x$, $y$ and $z$ code numbers obeying the relevant identity, as we described above. Next, we extend the translation recursively:
\begin{eqnarray*}
  (\neg\varphi)^*&=& \neg\varphi^* \\
  (\varphi\wedge\psi)^* &=& \varphi^*\wedge\psi^* \\
  (\exists x\, \varphi)^* &=& \possible\exists x\,\bigl(\bigwedge_{v\in\text{FV}(\varphi)} \!\!v\text{ represents a number }\wedge\varphi^*\bigr)\\
\end{eqnarray*}
In the exists case, we include the assertion that all the variables of the formula still represent numbers, since we only want to consider extensions in the which the other parts of the graph still represent numbers. This is where it is important that our number representation is absolute under extensions.

It now follows by induction that the truth of an arithmetic sentence $\varphi$ in $\N$ is equivalent to the truth of the modal graph translation $\varphi^*$ in any particular finite graph.
\end{proof}

Since one can encode finite graphs into arithmetic, it follows that we can also make a converse translation of finite modal graph theory into arithmetic, and so these theories are mutually interpretable. They are actually bi-interpretable, since each can see how it is that it is encoded into the translated copy of the other, although we shall not formulate a precise notion of bi-interpretation of models and potentialist systems here.

\begin{corollary}
Arithmetic is interpretable in modal graph theory, using the class of all graphs (not just finite graphs).
\end{corollary}

\begin{proof}
The point is that since theorem \ref{Theorem.Finiteness-is-expressible} shows that finiteness is expressible in modal graph theory, we can define the class of finite graphs within the class of all graphs, and therefore we can define the finiteness modal operators $\possible_{\textup{Fin}}$ and $\necessary_{\textup{Fin}}$ of \emph{true in some finite extension} and \emph{true in all finite extensions}, respectively.
\end{proof}

Let us consider next the case of countable modal graph theory. Let $H_{\omega_1}$ be the set of all hereditarily countable sets, the sets having a countable transitive closure. This is a model of $\ZFCm$, meaning Zermelo-Frankael set theory without the power set axiom (but see \cite{GitmanHamkinsJohnstone2016:WhatIsTheTheoryZFC-Powerset?} concerning a subtlety about the axiomatization). We shall show that truth in the hereditarily countable sets is interpretable in countable modal graph theory. Let us denote by $\Gamma\oplus\Lambda$ the disjoint sum of graphs $\Gamma$ and $\Lambda$. 

\begin{theorem}\label{Theorem.Hereditarily-countable-set-theory-is-interpretable}
Hereditarily countable set theory is interpretable in countable modal graph theory. We shall represent hereditarily countable sets with countable graphs and vertices and define a translation $\varphi\mapsto\varphi^*$ of set-theoretic assertions $\varphi$ to modal graph-theoretic assertions $\varphi^*$, such that
$$\<H_{\omega_1},\in>\satisfies\varphi(a_0,\dots,a_n)\quad\Iff\quad \Gamma_0\oplus\dots\oplus\Gamma_n\satisfies\varphi^*(\hat a_0,\ldots,\hat a_n),$$
where $(\Gamma_i,\hat a_i)$ is a countable graph and vertex representing the set $a_i$.
\end{theorem}

\begin{proof}
Every hereditarily countable set $a$ is an element of a countable transitive set $t$, such as the transitive closure of its singleton $\TC(\singleton{a})$. The structure $\<t,\in>$ is a countable set with a well-founded extensional relation $\in$, and every such well-founded extensional relation on a countable set is isomorphic to a countable transitive set via the Mostowski collapse.

We define that a \emph{graph code} for a hereditarily countable set is a graph consisting of a node $t$, whose neighbors are related by a well-founded and extensional relation $x\lperpdown y$, defined to hold when there are nodes between $x$ and $y$ as depicted here:
$$
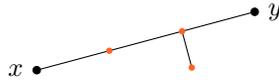
\begin{tikzpicture}
\draw[rotate=15] (0,0) node[dot=1pt,label=left:$x$] {} -- (1,0) node[Orange,dot] {} -- (2,0) node[Orange,dot] {} -- (3,0) node[dot=1pt,label=right:$y$] {}
      (2,0) -- (2,-.5) node[Orange,dot] {};
\end{tikzpicture}$$
In addition, there shall be a copy of $\omega$ in the sense of theorem \ref{Theorem.Countability-is-expressible} together with a bijection of the copy of $\omega$ with the neighbors of $t$ and the nodes used in the $\lperpdown$ relation. The property of being a graph code is expressible in modal graph theory, because the extensionality of $\lperpdown$ is expressible directly without modal operators by the assertion that distinct neighbors of $t$ have distinct sets of $\lperpdown$-predecessors amongst the neighbors of $t$, and the well-foundedness of $\lperpdown$ is expressible by the assertion that necessarily, in any extension in which the copy of $\omega$ is still a copy of $\omega$ and still provides a bijection with the neighbors of $t$ and the supplemental nodes used in the $\lperpdown$ relation, then there is no node whose neighbors are a set of neighbors of $t$ having no $\lperpdown$-minimal element. Note that we use the bijective copy with $\omega$ to ensure that in the graph extension, there are no new neighbors of $t$ and no new instances of the $\lperpdown$ relation; this is a form of absoluteness in our representation similar to what we had used in theorem \ref{Theorem.Arithmetic-is-interpretable} in the finite case. In addition, there will be a vertex $\hat x$ pointing at $t$ and at one of the neighbors $x$ of $t$. The graph, together with the vertex $\hat x$ code the set that would result from the image of $x$ under the Mostowski collapse of the $\lperpdown$ relation on the neighbors of $t$.

We can express that two graph codes $(\Gamma,x)$ and $(\Gamma',x')$ code the same set by asserting that possibly, there is a node pointing at nodes that are each adjacent to a node in $\Gamma$ and to a node in $\Gamma'$, as in the proof of theorem \ref{Theorem.Arithmetic-is-interpretable}, in such a way that they make a isomorphic correspondence between $\lperpdown$-downward closed subsets of $\Gamma$ and $\Gamma'$, which furthermore associates $x$ with $x'$. It will follows that the set coded by $x$ via $\Gamma$ will be the same as the set coded by $x'$ via $\Gamma'$. Similarly, we can express that one code $(\Gamma,x)$ codes an element of another $(\Lambda,y)$, if $(\Gamma,x)$ is isomorphic to the code formed by a $\lperpdown$-predecessor of $y$ in $\Lambda$. This tells us how to interpret the atomic formulas $x=y$ and $x\in y$ in terms of the codes. And we can simply extend the interpretation recursively through Boolean connectives and quantifiers as in theorem \ref{Theorem.Arithmetic-is-interpretable}, establishing the equivalence stated in the theorem by induction.
\end{proof}

\begin{corollary}
Second-order arithmetic is mutually interpretable with countable modal graph theory.
\end{corollary}

\begin{proof}
Second-order arithmetic is bi-interpretable with the structure $\<H_{\omega_1},\in>$, and this latter structure can define representations of any countable graph up to isomorphism. By the renaming lemma, the modal truths of any countable graph are the same as for the copies available in $H_{\omega_1}$, and so this structure can interpret countable modal graph theory.
\end{proof}

We could have used this set-theoretic representation even in the case of finite modal graph theory, representing every hereditarily finite set with a finite extensional relation $\lperpdown$ relation on a set together with a bijection to a finite cycle (thereby providing the absoluteness). This would show that $\<H_\omega,\in>$ is interpretable in finite modal graph theory.

But similarly, let us use the method on larger cardinals. The key requirement is a notion of stability for the cardinality in question. Let us define that a cardinal $\kappa$ is \emph{stably representable} in modal graph theory if there is a property $\phi$ expressible in modal graph theory with the following properties:
\begin{enumerate}
  \item There is a graph $G$ with a vertex $v$ satisfying $\phi(v)$.
  \item If $\phi(v)$ holds in a graph $G$ of a vertex $v$, then $v$ has exactly $\kappa$ many neighbors in $G$.
  \item The truth of $\phi(v)$ in $G$ depends only on the induced subgraph consisting of $v$ and its neighbors and the neighbors of the parameters (suppressed).
  \item If $\phi(v)$ holds in $G$ and also in an extension graph $H$, then $v$ has the same neighbors in $G$ as in $H$.
\end{enumerate}
This was the key property of our representation of $\omega$ that enabled the absoluteness feature we used in theorem \ref{Theorem.Hereditarily-countable-set-theory-is-interpretable}.

\begin{theorem}\label{Theorem.H_kappa+-is-interpretable}
If a cardinal $\kappa$ is stably representable in modal graph theory, then $\<H_{\kappa^+},\in>$ is interpretable in modal graph theory.
\end{theorem}

\begin{proof}
We can use the same method as in theorem \ref{Theorem.Hereditarily-countable-set-theory-is-interpretable}, except using a bijection with a stable representation of $\kappa$, rather than a copy of $\omega$. In this way, we represent sets in $H_{\kappa^+}$ using well-founded extensional relations on sets of size $\kappa$, with a distinguished vertex pointing at the set being represented. Once again, we can express the equivalence of codes and the set membership relation on the codes in the language of modal graph theory.
\end{proof}

Since $H_\kappa$ is definable inside $H_{\kappa^+}$, it follows that we can also interpret $\<H_\kappa,\in>$ in modal graph theory, when $\kappa$ is stably representable.

\begin{theorem}\label{Theorem.Stable-transfers-up}\
\begin{enumerate}
  \item If $\kappa$ is stably representable in modal graph theory, then so is $\kappa^+$.
  \item If $\kappa$ is stably representable in modal graph theory, then so is $2^\kappa$.
  \item If $\kappa$ is stably representable in modal graph theory, then so are $\aleph_\kappa$ and $\beth_\kappa$.
  \item If $\kappa$ is stably representable in modal graph theory, then so is the next $\beth$-fixed-point above $\kappa$.
  \item If $\kappa$ is stably representable in modal graph theory, then so is the first $\beth$-hyper-fixed-point above $\kappa$---the first cardinal $\lambda$ that is the $\lambda$th $\beth$-fixed-point above $\kappa$.
\end{enumerate}
\end{theorem}

\begin{proof}
Suppose that $\kappa$ is stably representable. We shall represent $\kappa^+$ by a vertex pointing at $\kappa^+$ neighbors that are well-ordered by the $\lperpdown$ relation, with the first $\kappa$ many of them fulfilling the stable representation of $\kappa$, and such that every node $u$ amongst the $\kappa^+$ beyond $\kappa$ is adjacent to a node that points at nodes forming a bijection between the $\lperpdown$-predecessors of $u$ and the $\kappa$ initial segment of the order. That is, the whole order does not have size at most $\kappa$ and not only do all these initial segments have size $\kappa$, but they come equipped already with the bijections witnessing that they have size $\kappa$. We can assert easily that $\lperpdown$ is a linear order; to assert that it is a well-order, we just have to claim that in every extension in which the copy of $\kappa$ remains stable and the $\lperpdown$-relation is still a linear order and the bijections of initial segments are still bijections with $\kappa$, there is no node pointing at a nonempty set with no $\lperpdown$-minimal element. If this is true, then it really must be a well-order, and if it is a well-order, then because of the stability with $\kappa$, none of the nodes in the order can gain new predecessors; and the $\kappa^+$ size order itself cannot gain new elements on top of all the previous nodes, because then those nodes would have $\kappa^+$ many predecessors and therefore could not exhibit a bijection with the set of size $\kappa$. So we have a stable copy of $\kappa^+$ and so it is stably representable.

To stably represent $2^\kappa$, we begin with a stable representation of $\kappa$, together with a node $P$ whose neighbors are each adjacent to distinct subsets of the set $\kappa$ many nodes, and such that furthermore, necessarily, in any extension in which the size $\kappa$ set fulfills its definition, if a node is adjacent to a subset of the set of size $\kappa$, then some neighbor of $P$ already realizes that same subset. The point is that if the neighbors of $P$ did not already represent the full power set, then we could extend by adding a node pointing out the missing the subset, and it would violate this necessity requirement. But if $P$ does represent the full power set, then any new pattern will already be represented. So we have a stable representation of $P(\kappa)$, which has size $2^\kappa$.

If $\kappa$ is stably representable, then we can stably represent a copy of $\kappa$ that is well-ordered by $\lperpdown$; this is simply a set of size $\kappa$ that is linearly ordered by $\lperpdown$, such that in every extension in which the copy of $\kappa$ is stable (so that the order has gained no new elements), there is no node pointing at a set having no $\lperpdown$-minimal element; and such that every proper initial segment of the order has size less than $\kappa$, which is to say, that in no extension in which $\kappa$ is preserved is there a bijection of $\kappa$ with an initial segment of the order. So we have now a stable representation of a well-order of length $\kappa$. To represent $\beth_\kappa$, we associate the least element of the $\kappa$-sequence with a copy of $\omega$, and at each successor node, we associate the power set of the previous node; and at limit nodes, we associate a set that is exhibited in bijection with the union of the prior sets. Since our copy of $\kappa$ is stable, we will have therefore iteratively and stably represented the iteration of the power set. Finally, we assert that there is a vertex whose neighbors are placed into bijection with all the associated sets at each stage; this will be a stable representation of $\beth_\kappa$. We can stably represent $\aleph_\kappa$ in a similar manner, simply by representing the successor cardinal at each step instead of the power set.

We now iterate this to find a stable representation of the next $\beth$-fixed point. Consider the least $\beth$-fixed point, which is the cardinal $\theta=\sup\theta_n$, where $\theta_0=\omega$ and $\theta_{n+1}=\beth_{\theta_n}$. To stably represent this, we have a copy of $\omega$ pointing out nodes to represent each $\theta_n$, each with a well-order, and then at each $\theta_{n+1}$ we assert that the previous construction has been already arranged, so that the node representing $\theta_{n+1}$ is pointing out a set of size $\beth_{\theta_n}$, along with a well-order of it. The point is that this whole iteration is sufficiently uniform that we can describe that a graph has carried out the whole construction to reach the limit node, which is the $\beth$-fixed point. To find the next $\beth$-fixed point after a given stably representable cardinal, we simply start with $\theta_0$ representing $\kappa+1$, rather than $\omega$.

Similarly, since the process of representing the next $\beth$-fixed point is stably representable, we can represent the process of iterating it sufficiently to find the first hyper-fixed-point.
\end{proof}

Theorem \ref{Theorem.Stable-transfers-up} has consequences for the possibility or impossibility of \Lowenheim-Skolem conclusions in modal graph theory. Because these various sizable cardinals are expressible in modal graph theory, they form bounds on the sizes of possible elementary substructures and extensions. A comparatively small graph, after all, cannot be elementary in a large one in modal graph theory, if the size of the small graph is expressible in the language of modal graph theory.

The theorem shows that the set-theoretic structures $H_\omega$, $H_{\omega_1}$, $H_{\omega_2}$, $V_{\omega^{\omega^\omega}+5}$ and so on are all interpretable in modal graph theory. And indeed much more.

\begin{corollary}
Set-theoretic truth in $V_\theta$, where $\theta$ is the first $\beth$-hyper-fixed point, is interpretable in modal graph theory.
\end{corollary}

\begin{proof}
The theorem shows that the first $\beth$-hyper-fixed point is stably representable, and so theorem \ref{Theorem.H_kappa+-is-interpretable} shows that $V_\theta$, which is the same as $H_\theta$, is interpretable in modal graph theory.
\end{proof}

And we may proceed similarly to the next $\beth$-hyper-fixed point, the next hyper-hyper-fixed point, and so on. These will all also be stably representable, and so we can interpret set-theoretic truth for a quite a long way into the cumulative hierarchy. In this sense, modal graph theory can serve as a foundation of mathematics. 

Does it ever stop?

\begin{question}\label{Question.Set-theory-interpreted-in-modal-graph-theory?}
Is set-theoretic truth, that is, truth in the full set-theoretic universe $V$, interpretable in modal graph theory?
\end{question}

The key obstacle concerns the question of whether we can find sufficient stable representations of any given set. It is not clear how to translate definable cardinals in \ZFC, such as the first $\Sigma_3$-correct cardinal, into stable representations in modal graph theory.

There is a related question that we have been unable to resolve. In any graph $G$, let us denote by $G_x$ the set of neighbors adjacent to a given node $x$.
\begin{itemize}
  \item (Equinumerosity problem) Can we express in the language of modal graph theory that the neighbor set $G_x$ of vertex $x$ is equinumerous with the neighbor set $G_y$ of vertex $y$?
  \item (Cardinal comparability problem) Can we express in the language of modal graph theory that the neighbor set $G_x$ of vertex $x$ has cardinality less than or equal to the neighbor set $G_y$ of vertex $y$?
\end{itemize}

We conjecture that the answers are both negative. A positive answer to the cardinal comparability problem, of course, will imply a positive answer to the equinumerosity problem.

If it should turn out that contrary to our conjecture, the cardinal comparability problem has a positive solution, then we would like to mention that we will be able to express in modal graph theory that a relation such as $\lperpdown$ is well-ordered. That is, if we have a node whose neighbors are linearly ordered by the $\lperpdown$ relation, then the relation is well-ordered if and only if it is possible that every node in the order is adjacent to a set of neighbors, such that the cardinality comparisons of those neighbor sets agrees with the $\lperpdown$ order. The point is that in \ZFC, any set of cardinal sizes is well-ordered by comparability (this fact is equivalent to the axiom of choice).

This observation seems important, since the ability to express the well-order concept in modal graph theory seems likely to be highly relevant for the capacity of modal graph to interpret set-theoretic truth.

\section{Actuality operator @}\label{Section.Actuality}

Let us now discuss the fruitful extension of the modal language by means of the \emph{actuality} operator $@$, which allows one to refer to the actual world and more generally to the various worlds that are in effect referenced during the course of interpreting a modal statement. Let us explain the usage and semantics by means of several examples. The statement $\necessary(\varphi\to(\psi\iff @\psi))$, expressing the assertion that ``necessarily, if $\varphi$ is true, then $\psi$ holds if and only if $\psi$ is actually true.'' This is true in a world $w$, when in every world $u$ that $w$ can access, if $\varphi$ is true in $u$, then $\psi$ is true in $u$ if and only if $\psi$ is true in $w$. Thus, the @ operator in effect references truth in the original world $w$. Let us introduce the predication $@x$ to mean that individual $x$ is actual, that is, that $@\exists u\ u=x$, which asserts that $x$ is an individual of the actual world. Similarly, we introduce the quantifier $\forall@ x\ \varphi$ to mean that $\forall x\,(@x\to\varphi)$, or more simply, that $\varphi$ holds (in the current possible world of evaluation) for all actual $x$, that is, for all $x$ in the actual world. And the quantifier $\exists @x\varphi$ means $\exists x\,(@x\wedge\varphi)$, asserting that there is some actual $x$ fulfilling $\varphi$ in the current possible world of evaluation.

For example, in the modal language of graph theory with actuality, the assertion
$$\possible\exists x\forall y\left(x\sim y\iff (@y\wedge @\forall z\ \neg y\sim z)\right)$$
expresses that it is possible that there is a node adjacent to all and only the isolated nodes of the actual world.

When the modal depth of the assertion increases, there are several worlds at play in the semantics of the assertion, and a more complex usage and semantics for @ allows us to refer to them directly. The unadorned @ refers as above to the world in which the overall statement is being evaluated, and otherwise the subscripts $@_1$, $@_2$ and so on refer to the worlds referenced by corresponding subscripts on the modal operators. For example, the assertion
$$(\possible\exists x\forall @y\ Txy)\wedge\necessary_1\possible\exists x\forall @_1y\ Txy$$
asserts that possibly someone is taller than every actual person, and furthermore, this is necessarily true. The second clause asserts that in every accessible world, it is possible that someone is taller than everyone in \emph{that} world, not just in the original world. The operator $\necessary_1$ here is not a different modal operator than $\necessary$, for the two necessities themselves have the same meaning; rather, the subscript on $\necessary_1$ tells us to which world $@_1$ will now refer inside the scope of $\necessary_1$.

To illustrate further, consider the assertion that necessarily there are two possible worlds, such that every individual who is tall in one of them is short in the other. One might be tempted at first to formulate it like this:
$$\necessary\bigl[\possible_1\forall x(Tx\to @_2Sx)\wedge\possible_2\forall x(Tx\to @_1Sx)\bigr].$$
This expression, however, is not well formed; the semantics of it are not compositional, since the $@_1$ does not occur under the scope of $\possible_1$ and $@_2$ does not occur under the scope of $\possible_1$. One can, however, properly express exactly the desired statement like this:
$$\necessary_3\possible_1@_3\possible_2 \bigl[@_1\forall x(Tx\to @_2Sx)\wedge @_2\forall x(Tx\to @_1Sx)\bigr].$$
In this way, the actuality operator allows for a rich expression of truth and possibility relations between worlds.

In the potentialist system of all graphs, the assertion
$$\possible_1\bigl[\exists x\forall y(x\sim y\iff @y)\wedge\possible\exists z\forall y(z\sim y\iff (@_1y\wedge \neg @y))\bigr]$$
is true at a graph $G$ exactly when there is a larger graph $G_1$ with a node $x$ that is connected to all and only the nodes of $G$ and a further graph $G_2$ with a node $z$ connected to all and only the nodes of $G_2$ that are in $G_1$ but not in $G$.

\begin{theorem}
Both the equinumerosity problem and the cardinal comparability problem are expressible in the language of modal graph theory with actuality $\mathcal{L}^{\possible,@}_\sim$.
\end{theorem}

\begin{proof}
It suffices to express only the comparability problem. For any graph $G$, let $G_v$ denote the nodes in $G$ adjacent to $v$. The relation $|G_v|\leq|G_w|$ is expressible by saying: there is some larger graph $H$ with a node that points at nodes that index an injective correspondence from the \emph{actual} neighbors of $v$ to the \emph{actual} neighbors of $w$. This is the same idea as used with $\leq$ in the proof of theorem \ref{Theorem.Arithmetic-is-interpretable}.
\end{proof}

The main point we should like to make is that modal graph theory with the actuality operator can interpret full set-theoretic truth.

\begin{theorem}\label{Theorem.Set-theoretic-truth-is-interpretable}
The modal logic of graph theory with actuality can interpret set-theoretic truth. There is a translation of set-theoretic assertions to modal graph assertions $\varphi\mapsto\varphi^*$, such that a set-theoretic sentence $\varphi$ is true in the set-theoretic universe $(V,\in)$ if and only if $\varphi^*$ is true in the empty graph, or in any particular graph.
\end{theorem}

\begin{proof}
This theorem follows from the methods used in proving the theorems of section \ref{Section.Interpretative-power-of-modal-graph-theory}. The main point is that the actuality operator @ will enable us avoid the need for stable representations, since with @ we can directly refer to the desired coded structure.

So we may code any set with a node whose neighbors form a set that under the $\lperpdown$ relation is well-founded and extensional. And so any set is represented by pointing to such a well-founded extensional relation and a point in it. As in theorem \ref{Theorem.H_kappa+-is-interpretable}, we can define equivalence of codes and the set membership relation. And we can translate set-theoretic assertions into modal graph theory by extending this interpretation of the atomic truths through the Boolean connectives, and interpreting the set-theoretic $\exists x$ as ``$\possible\exists x$ such that $x$ codes a set,'' just as in section \ref{Section.Interpretative-power-of-modal-graph-theory}.
\end{proof}

There are some subtleties about the metamathematical status of this theorem, since we cannot refer to set-theoretic truth inside set theory. One can interpret the theorem in \ZFC\ as a theorem scheme, a separate claim about interpreting $\Sigma_n$ set-theoretic truth in modal graph theory, for each metatheoretic $n$, but always using the same interpretation method. Alternatively, we can interpret the theorem as a true theorem in an extension of \ZFC\ to a theory such as \KM\ or just $\GBC+\ETR_\omega$, which proves the existence of first-order set-theoretic truth predicates.

\begin{question}
Is actuality @ expressible in modal graph theory?
\end{question}

We conjecture not, but we do not know how to prove this. One idea may be to show that the equinumerosity problem is not expressible in the language of modal graph theory---perhaps this can be proved by means of modal pebble games. If successful, this would show actuality is not expressible in modal graph theory, because with actuality, we can express equinumerosity.

\section{Set-theoretic and meta-mathematical issues}\label{Section.Metamathematical-issues}

Although we have claimed to define the semantics for modal truth in $\Mod(T)$, the meta-mathematics of this definition involves some set-theoretic subtleties. The basic problem is that $\Mod(T)$ is a proper class and the recursive definition of truth is not a set-like recursion. Therefore, we can't seem to undertake the definition legitimately in \ZFC. But in Kelley-Morse set theory, for example, or even merely in \Godel-Bernays set theory with the axiom of elementary transfinite recursion $\ETR$, or even just $\ETR_\omega$, then we can prove that there is a solution of the recursive definition of modal truth. We refer the readers to \cite{GitmanHamkinsHolySchlichtWilliams:The-exact-strength-of-the-class-forcing-theorem} for further discussion of the role of \ETR\ in defining truth predicates.

\begin{question}
In \ZFC\ can one define the satisfaction relation for modal graph theory for the class of all graphs?
\end{question}

This question is related to the question whether modal graph theory interprets set-theoretic truth. If it does, then the answer to this question will be negative, since by Tarski's theorem on the non-definability of truth one cannot define first-order set-theoretic truth within first-order set theory.

\begin{theorem}
No \ZFC-definable class defines the satisfaction relation for modal graph theory with actuality, that is, a truth predicate for the class of graphs in $\mathcal{L}_\sim^{\possible,@}$.
\end{theorem}

\begin{proof}
By theorem \ref{Theorem.Set-theoretic-truth-is-interpretable}, set-theoretic truth is interpretable in modal graph theory with actuality. So by Tarski's theorem on the non-definability of truth, there can be no definable truth predicate.
\end{proof}

Meanwhile, in the stronger class-based theories such as \KM\ or $\GBC+\ETR$, there is a satisfaction class for first-order set-theoretic truth, and with this we can define the satisfaction relation for modal graph theory with actuality. To our way of thinking, the necessary set-theoretic difficulty of defining the modal semantics for a potentialist system such as $\Mod(T)$ poses a philosophical difficulty for advocates of potentialism seeking a simplified, reduced ontology, for it is as difficult to define the modal semantics for modal graph theory with actuality, for example, even when one's ontology is reduced to having only a set-sized graph at a time, as it is to define the semantics for full-blown set-theoretic truth. Even \ZFC\ cannot do it. In this sense, it seems impossible for the potentialists to have the parsimonious ontology they seek, if they also wish to have a coherent potentialist semantics.


\bibliographystyle{halpha}
\bibliography{MathBiblio,HamkinsBiblio,WebPosts,references}

\end{document}